\documentclass[a4paper,twoside,11pt,reqno]{amsart}
\usepackage[margin=1in]{geometry}
\usepackage[english]{babel}
\usepackage[utf8]{inputenc}
\usepackage{amsmath}
\usepackage{amssymb}
\usepackage{amsfonts}
\usepackage{amsthm}
\usepackage{bm}
\usepackage{mathrsfs}
\usepackage{hyperref}
\usepackage{comment}
\hypersetup{
	colorlinks=true,
	linkcolor=blue,
	filecolor=blue,
	citecolor=purple,      
	urlcolor=cyan,
}
\usepackage{xcolor}
\usepackage{dsfont}
\usepackage{physics}
\usepackage{stmaryrd}
\usepackage{enumitem}

\usepackage{tikz-cd}
\usepackage{tikz}
\usepackage{tikzsymbols}
\usepackage{mathtools}
\usetikzlibrary{matrix}

\usepackage[matrix,arrow,cmtip]{xy} 
\SelectTips{cm}{}
\newdir{(}{{}*!/-5pt/@^{(}}
\newdir{(x}{{}*!/-5pt/@_{(}}
\newdir{+}{{}*!/-9pt/{}}
\newdir{>+}{@{>}*!/-9pt/{}}
\entrymodifiers={+!!<0pt,\fontdimen22\textfont2>}

\theoremstyle{plain}
\newtheorem{theorem}[subsubsection]{Theorem}

\newtheorem{lemma}[subsubsection]{Lemma}
\newtheorem{proposition}[subsubsection]{Proposition}

\newtheorem{corollary}[subsubsection]{Corollary}
\newtheorem{conjecture}[subsubsection]{Conjecture}
\newtheorem{properties}[subsubsection]{Properties}
\theoremstyle{definition}
\newtheorem{definition}[subsubsection]{Definition}
\newtheorem{definition-proposition}[subsubsection]{Definition/Proposition}

\newtheorem{question}[subsubsection]{Question}
\newtheorem{example}[subsubsection]{Example}
\theoremstyle{remark}
\newtheorem{remark}[subsubsection]{Remark}

\numberwithin{equation}{section}

\newcommand{\CC}{\mathbb{C}}
\renewcommand{\AA}{\mathbb{A}}

\newcommand{\QQ}{\mathbb{Q}}
\newcommand{\ZZ}{\mathbb{Z}}

\newcommand{\HH}{\mathrm{H}}

\newcommand{\Pic}{\mathrm{Pic}}

\newcommand{\Br}{\mathrm{Br}}
\newcommand{\PP}{\mathbb{P}}

\newcommand{\CH}{\mathrm{CH}}

\newcommand{\Gm}{\mathbb{G}_{m}}

\newcommand{\cH}{\mathcal{H}}

\newcommand{\Spec}{\mathrm{Spec}}

\title{Vector bundles on rational topologically contractible affine threefolds}
\date{}
\author{Haoyang Liu and Biman Roy}
\address{University of California, Santa Barbara}
\email{haoyangliu@ucsb.edu}
\address{University of Southern California}
\email{bimanroy@usc.edu}

\begin{document}

\begin{abstract}
The generalized Serre question asks whether every algebraic vector bundle on a topologically contractible smooth affine complex variety is trivial. We give an affirmative answer for rational threefolds. More generally, for a topologically contractible smooth affine complex threefold $X$, we prove that $\CH^2(X)=0$ whenever $X$ admits a smooth projective compactification whose Chow group of $0$-cycles is supported on a curve. This uncovers the link between the generalized van de Ven question, Bloch's conjecture and the generalized Serre question for threefolds. We also prove that every Koras--Russell threefold is rational and therefore has only trivial algebraic vector bundles, hence answer a question of Koras and Russell.
\end{abstract}
	\maketitle
    \tableofcontents
\section{Introduction}
In classical topology, for a CW complex $X$, the projection map $X \times I \to X$ induces a bijection between the isomorphism classes of rank-$n$ vector bundles
$$\operatorname{Vect}_n^k(X) \xrightarrow{\sim} \operatorname{Vect}_n^k(X \times I),$$
where $I = [0,1]$ is the closed unit interval and $\operatorname{Vect}_n^k(Y)$ denotes the set of all isomorphism classes of rank $n$ real or complex vector bundles over a topological space $Y$ (with $k = \mathbb{R}$ or $\mathbb{C}$). Moreover, the set-valued functor $X \mapsto \operatorname{Vect}^k_n(X)$ on the category of CW complexes is representable in the homotopy category of CW complexes. Specifically, the pullback of the universal bundle over the infinite Grassmannian $G_n(k^\infty)$ induces a bijection
$$
[X, G_n(k^\infty)] \xrightarrow{\sim} \operatorname{Vect}^k_n(X),
$$
where $[X, Y]$ denotes the set of homotopy classes of maps from $X$ to $Y$. Consequently, any real or complex (topological) vector bundle over a topologically contractible CW complex is trivial. More generally, this topological result also holds over any paracompact topological space.

Let $k$ be a field, and let $\mathbf{Sm}/k$ denote the category of smooth varieties over $k$. The situation in algebraic geometry differs from the topological picture. Here, the counterpart of the unit interval $[0,1]$ is the affine line $\mathbb{A}^1_k = \Spec\ k[t]$. The set-valued functor associating to a variety $X \in \mathbf{Sm}/k$ the set $\mathcal{V}_r(X)$ of isomorphism classes of algebraic vector bundles of rank $r$ over $X$ is not $\mathbb{A}^1$-invariant on $\mathbf{Sm}/k$. For instance, there exists a rank-$2$ bundle over $\mathbb{P}^1_{k} \times_{k} \mathbb{A}^1_{k}$ that is not the pullback of any vector bundle over $\mathbb{P}^1_{k}$ \cite[Example 5.2.4.7]{Asok2021}. The celebrated Quillen--Suslin theorem says that any algebraic vector bundle over $\mathbb{A}^n_k$ is trivial. Lindel proved that if the base $X \in \mathbf{Sm}/k$ is an affine variety, then the canonical projection $X \times_{k} \mathbb{A}^1_{k} \to X$ induces a bijection
$$
\mathcal{V}_r(X) \xrightarrow{\sim} \mathcal{V}_r(X \times_{k} \mathbb{A}^1_{k}).
$$

A smooth complex variety naturally carries the structure of a smooth complex manifold. Consequently, any topological vector bundle over a topologically contractible smooth complex variety is trivial. Inspired by this, the generalized Serre question asks the following:
\begin{question} (\cite[\S 8]{ZaidenbergSurvey}; \cite[Question 6, Chapter 5]{Asok2021}) \label{serre question}
   Let $X$ be a topologically contractible smooth affine variety over $\mathbb{C}$. Is every algebraic vector bundle over $X$ trivial?
\end{question}
Here the affineness assumption is crucial; there is a topologically contractible smooth strictly quasi-affine complex fourfold that admits a nontrivial rank-$2$ algebraic vector bundle \cite[Example 5.3.5.4]{Asok2021}. Fujita showed that topologically contractible smooth complex surfaces are affine \cite{Fujita1982OnTT}. However, it is not known whether every topologically contractible complex threefold is affine \cite[Question 5.3.5.5]{Asok2021}.

The Chow groups of algebraic cycles play an important role in the theory of algebraic vector bundles. The $i$-th Chern class associates to an algebraic vector bundle on $X \in \mathrm{Sm}/k$ an element of the Chow group $\CH^i(X)$ of codimension-$i$ cycles. For $X \in \mathrm{Sm}/k$, the projection map $X \times_k \mathbb{A}^n_k \to X$ induces an isomorphism
$$\CH^i(X) \to \CH^i(X \times_k \mathbb{A}^n_k).$$
To determine the triviality of algebraic vector bundles over a complex affine variety of dimension $d$, Serre's splitting principle allows us to consider only algebraic vector bundles of rank at most $d$. Any topologically contractible smooth complex variety has trivial Picard group \cite[Theorem 5.5.2.7]{Asok2021}, which coincides with $\CH^1(X)$. If $X$ is a smooth affine complex threefold, Asok--Fasel and Mohan Kumar--Murthy showed that the Chern class map
$$\mathcal{V}_2(X) \to \CH^1(X) \times \CH^2(X)$$
is a bijection \cite[Theorem 5.5.2.8]{Asok2021}. Gurjar and Shastri \cite{GSa, GSb} subsequently proved that topologically contractible smooth complex surfaces are indeed rational.
This in particular implies that the generalized Serre question has an affirmative answer in dimensions $\leq 2$.

\begin{theorem}\cite[Theorem 5.5.2.8]{Asok2021}
    Let $X$ be a topologically contractible smooth complex variety of dimension at most $2$. Then every algebraic vector bundle over $X$ is trivial.
\end{theorem}
More generally, Asok showed that a topologically contractible smooth complex surface has Voevodsky motive isomorphic to the motive of a point \cite[Theorem 1]{Asok_motive}.
Mohan Kumar--Murthy showed that the Chern class map
$$\mathcal{V}_3(X) \to \CH^1(X) \times \CH^2(X) \times \CH^3(X)$$
is a bijection \cite[Theorem 2.1(iii)]{MohanKumarMurthy}. Thus, in dimension $3$, the following criterion is known:
\begin{theorem} \cite[Theorem 5.5.2.10]{Asok2021}\label{threefold}
    Let $X$ be a topologically contractible smooth affine threefold over $\mathbb{C}$. Then every algebraic vector bundle over $X$ is trivial if and only if $\CH^2(X)$ and $\CH^3(X)$ are trivial.
\end{theorem}
However, for a topologically contractible smooth complex variety of dimension greater than $2$, it is not known whether its Chow groups, or more generally its Voevodsky motive, are trivial.
Thus, the generalized Serre question is open in higher dimensions. Unless otherwise stated, we assume that the base field is $\mathbb{C}$.

A variety over $\mathbb{C}$ is called rational if it is birational to the projective space $\mathbb{P}^n_\mathbb{C}$. In dimension $1$, the affine line $\mathbb{A}^1_\mathbb{C}$ is the only topologically contractible smooth complex curve. Gurjar and Shastri \cite{GSa, GSb} proved that topologically contractible smooth complex surfaces are indeed rational. However, the situation in dimensions three and higher is not known in general.
In this context, the generalized van de Ven question asks the following: 
\begin{question} \cite[Question 8]{Asok2021} \label{van question}
   Let $X$ be a topologically contractible smooth complex variety. Is $X$ rational?
\end{question}

There are also several important classes of ``nearly rational" varieties, namely, stably rational, retract rational, and unirational varieties.
All these varieties belong to a general class, namely rationally connected varieties. A smooth proper variety $X$ is called rationally connected if any two general points of $X$ can be joined by a $\mathbb{P}^1$ in $X$. Following the convention used here for non-proper varieties, a smooth variety is called rationally connected if it admits a smooth projective rationally connected compactification; this condition is independent of the chosen smooth projective compactification. If the dimension of $X$ is at most $2$, then rational connectedness is equivalent to rationality.
A rationally connected variety $X$ is uniruled; that is, $X$ admits a dominant rational map \cite[Section 1.1, Chapter IV]{Kollar}
$$\mathbb{P}^1 \times Y \dashrightarrow X \quad \text{with} \quad \dim(Y) = \dim(X) - 1.$$

Several cohomological invariants have been studied to detect rationality properties. Artin--Mumford constructed a unirational threefold $X$ that is not rational \cite{ArtinMumford} by showing that the torsion subgroup of $H^3(X, \mathbb{Z})$ is nontrivial. They showed that $H^3(X, \mathbb{Z})_{\mathrm{tors}}$ is a birational invariant for smooth projective varieties $X$, and this group is related to
the Brauer group $\Br(X):= H^2_{\acute{e}t}(X, \mathbb{G}_m)$, which is a torsion group.
However, the Artin--Mumford example cannot contain any topologically contractible affine open subvariety (Corollary \ref{contractible trivial Brauer group}). This suggests that the generalized van de Ven question may admit a positive answer for threefolds.

Hodge theory plays an important role in the study of projective varieties. A smooth complex projective variety $X$ is a K\"ahler manifold. The singular cohomology $H^n(X, \mathbb{C})$ admits a canonical Hodge decomposition
$$ H^n(X,\mathbb{C})= \bigoplus_{p+q=n}H^{p,q}(X),
H^{p,q}(X)\cong H^q(X, \Omega^p_X),
h^{p,q}(X)= \dim_\mathbb{C}(H^{p,q}(X)).$$
Here $\Omega^p_X$ is the sheaf of holomorphic $p$-forms on $X$ \cite[Chapter 6]{voisin1}. The integers
$$P_{m,n}(X):= h^0(X, (\Omega_X^n)^{\otimes m}) = \dim_\mathbb{C}\Gamma(X,(\Omega_X^n)^{\otimes m})$$
are important birational invariants of a smooth projective variety $X$ over $\mathbb{C}$ for integers $m,n \geq 0$.

If $X$ is rationally connected, then all $P_{m,n}$ vanish for $m,n \geq 1$. Castelnuovo's rationality criterion states that a smooth projective surface $S$ is rational if and only if $q(S)=h^1(S,\mathcal O_S)=0$ and $P_{2,2}(S)=0$. A projective variety $X$ of dimension $d$ for which the integers $P_{m,d}$ vanish for every $m \geq 1$ is said to have negative Kodaira dimension. An uniruled variety has negative Kodaira dimension. In dimension $3$, the converse follows from three-dimensional minimal model theory; see \cite{Kollar}.

    A smooth affine variety $X$ over $\mathbb{C}$ admits mixed Hodge structure \cite{PMIHES_1971__40__5_0}. Therefore, if $\bar{X}$ is a smooth projective compactification of $X$ such that $X$ is topologically contractible and $\bar{X} \setminus X$ is a divisor with simple normal crossings, then the Hodge numbers $h^{p,0}(\bar{X}) = 0$, for $p \geq 1$ (Theorem \ref{regular q-forms of contractible}). On the other hand, if $\bar{X}$ is rationally connected, then $\CH_0(\bar{X}) \cong \mathbb{Z}$. More generally, the generalized Mumford theorem \cite[Theorem 3.13]{voisinannals} implies that if $\CH_0(\bar{X})$ is isomorphic to $\mathbb{Z}$, then $h^{p,0}(\bar{X}) = 0$. Bloch's conjecture asks for the converse:
    \begin{conjecture} \cite[Conjecture 1.11]{voisinannals} \label{bloch}
      Let $X$ be a smooth projective variety over $\mathbb{C}$ such that $h^{p,0}(X) = 0$ for every $p \geq 1$. Then $\CH_0(X) \cong \mathbb{Z}$.
    \end{conjecture}
Assume in addition that the affine variety $X$ is topologically contractible. Then $\bar{X}$ is simply connected (Theorem \ref{simply connected}). Therefore, $\CH_0(\bar{X}) \cong \mathbb{Z}$ is equivalent to the fact that $\CH_0(\bar{X})$ is supported on a curve (Corollary \ref{main corollary}). The main theorem of this article is the following:
\begin{theorem} (Theorem \ref{main theorem})
    Let $X$ be a topologically contractible smooth affine complex threefold that admits a smooth projective compactification $\bar{X}$ such that $D=\bar{X} \setminus X$ is a divisor with simple normal crossings. Assume also that $\CH_0(\bar{X})$ is isomorphic to $\mathbb{Z}$. Then $\CH_1(X)$
    is trivial.
    \end{theorem}
    Therefore, using the Chow-theoretic criteria for the triviality of vector bundles, we obtain the following partial answer to the generalized Serre question in dimension $3$:

\begin{corollary}(Corollary \ref{main corollary})
    Let $X$ be a topologically contractible smooth affine complex threefold that admits a smooth projective compactification $\bar{X}$ such that $D=\bar{X} \setminus X$ is a divisor with simple normal crossings. Assume also that $X$ satisfies one of the following properties:
    \begin{enumerate}
    \item $\CH_0(\bar{X})$ is supported on a curve (or equivalently, $\CH_0(\bar{X}) \cong \mathbb{Z}$).
      \item $\bar{X}$ is rationally connected (for example, $\bar{X}$ is a Fano variety).
 \end{enumerate}       
        Then every algebraic vector bundle over $X$ is trivial.
         \end{corollary}

The following corollaries link the generalized Serre question (Question \ref{serre question}) to the generalized van de Ven question (Question \ref{van question}) and Bloch's conjecture (Conjecture \ref{bloch}):
\begin{corollary}(Corollary \ref{vdVimpliesSerre})
   An affirmative answer to the generalized van de Ven question implies an affirmative answer to the generalized Serre question in dimension $3$, i.e. 
   \[
\text{generalized van de Ven in dimension }3
\quad\Longrightarrow\quad
\text{generalized Serre in dimension }3;
\]
\end{corollary}

\begin{corollary}(Corollary \ref{BlochimpliesSerre})
    Let $X$ be a topologically contractible smooth affine complex threefold that admits a smooth projective compactification $\bar{X}$ such that $D=\bar{X} \setminus X$ is a divisor with simple normal crossings. Then an affirmative answer to Bloch's conjecture for threefolds implies an affirmative answer to the generalized Serre question, i.e. \[
\text{generalized Bloch for the compactification}
\quad\Longrightarrow\quad
\text{generalized Serre in dimension }3;
\]
\end{corollary}

It is worth mentioning that all classes of topologically contractible affine threefolds
considered in this article are rational. Therefore, it is natural to ask the following (see also Question \ref{van unramified question} for its relation to the van de Ven question):
\begin{question} \label{version serre question}
    Let $X$ be a topologically contractible smooth affine complex threefold that admits a smooth projective compactification $\bar{X}$ such that $D=\bar{X} \setminus X$ is a divisor with simple normal crossings. Is $\bar{X}$ rationally connected? Note that an affirmative answer to this question implies an affirmative answer to the generalized Serre question.
\end{question}

Affine modification (\cite{Zaidenbergmodification}, \cite{DPO}) is an extremely useful technique for constructing a topologically contractible variety from affine space. In particular, the main theorem implies that every algebraic vector bundle over such a threefold is trivial.
\begin{corollary} (see also Section \ref{appendix})
    Let $X$ be a smooth affine topologically contractible threefold over $\mathbb{C}$ that is also rational. Suppose that $\Tilde{X}$ is an affine modification of $X$ along $C \subset D$ such that both $C$ and $D$ are contractible ($C$ is contained in the smooth locus of $D$; $D$ is not necessarily smooth, as in \cite{Zaidenbergmodification} or \cite{DPO}). Then every algebraic vector bundle over $\Tilde{X}$ is trivial.
\end{corollary}
We discuss some examples of affine modifications in this context in Section \ref{appendix}. Following \cite[Corollary 2.6]{kaliman2002bfcactionscontractiblethreefolds}, we obtain the following class of examples covered by the main theorem:
\begin{corollary}(Corollary \ref{Cplusaction})
    Let $X$ be a smooth affine topologically contractible threefold over $\mathbb{C}$ with a nontrivial algebraic $\CC^+$-action. Then every algebraic vector bundle over $X$ is trivial.
\end{corollary}
Koras--Russell threefolds form an important family of smooth
topologically contractible affine complex threefolds that are not
isomorphic to $\mathbb A^3_{\mathbb C}$. We prove that every
Koras--Russell threefold is rational
(Theorem \ref{all KR rational}), which also provides another evidence for the generalized van de Ven question for threefolds. Therefore, the main theorem implies the following:
\begin{corollary}(Corollary \ref{korasrussellthird})
   Every algebraic vector bundle over a Koras–Russell threefold is trivial.
\end{corollary}

Corollary \ref{korasrussellthird} give a complete answer to \cite[Question 7.12]{KorasRussell}. For Koras--Russell threefolds of the first and second kinds, this triviality was already known from
\cite[Corollary 3.8]{Murthy02} and
\cite[Corollary 3.7]{HKO}. For the displayed family
$Y(d,\alpha_1,\alpha_2,\alpha_3)$, an independent proof was
recently obtained in
\cite{syed2026vectorbundlescertainkorasrussell}. For sufficiently large exponents, the third-kind family contains threefolds of non-negative logarithmic Kodaira dimension and hence without a non-trivial $\mathbb G_a$-action. 

Suppose that the contractible threefold $X$ admits a smooth projective compactification $\bar{X}$ such that
the boundary $D = \bar{X} \setminus X$ 
    is a divisor with simple normal crossings. In the Picard-rank-one Fano setting considered in \cite[Corollary 2.1]{Kishimoto2005}, the one-component case forces $X$ to be isomorphic to $\mathbb{A}^3$. There is also a complete classification of the pairs $(\bar{X},D)$ satisfying the additional hypotheses of \cite[Main Theorem]{MS1990} or \cite[Theorem 1.1]{Kishimoto2005}, where $D$ consists of two smooth irreducible components. In the cases of \cite{MS1990}, the open threefold $\bar{X}\setminus D$ is biholomorphic to $\mathbb{C}^3$. We show that all varieties appearing in \cite[Table 1]{MS1990} or \cite[Table 1]{Kishimoto2005} have trivial Voevodsky motive. Therefore, the generalized Serre question has an affirmative answer for these classified families.
\begin{theorem}(Theorems \ref{trivial motive muller} and \ref{trivial motive Kishimoto})
   Suppose that $X, \bar{X}, D$ are as in the previous paragraph (as in the list \cite[Table 1]{MS1990} or \cite[Table 1]{Kishimoto2005}); in particular, the boundary $D= \bar{X} \setminus X$ consists of two smooth projective surfaces.
   Then $M(X)$ is isomorphic to $\mathbb{Z}$ in $\textbf{DM}_{gm}(\mathbb{C}, \mathbb{Z})$, and hence in this case every vector bundle over $X$ is trivial.
\end{theorem}

Morel and Voevodsky \cite{MV99} established $\mathbb{A}^1$-homotopy theory, synthesizing algebraic geometry with homotopy-theoretic methods. They considered the Nisnevich topology on $\mathbf{Sm}/k$ and the category of simplicial presheaves $\Delta^{\mathrm{op}}\mathbf{PSh}(\mathbf{Sm}/k)$ on $\mathbf{Sm}/k$, which carries the projective model structure. Using model category techniques, they constructed the unstable $\mathbb{A}^1$-homotopy category $\mathbf{H}(k)$ by inverting the Nisnevich local weak equivalences and the projection maps $X \times_{k} \mathbb{A}^1_{k} \to X$ for $X \in \mathbf{Sm}/k$. By construction, $X$ and $X \times_k \mathbb{A}^1_k$ are isomorphic in $\mathbf{H}(k)$. A variety $X \in \mathbf{Sm}/k$ is called $\mathbb{A}^1$-contractible if the structure morphism $X \to \operatorname{Spec} k$ is an isomorphism in $\mathbf{H}(k)$. An $\mathbb{A}^1$-contractible complex variety is also topologically contractible. However, topological contractibility is weaker than $\mathbb{A}^1$-contractibility \cite[Corollary 1.3]{ChoudhuryRoy+2024+55+80}. By construction, the affine spaces $\mathbb{A}^n_k$ are $\mathbb{A}^1$-contractible. The isomorphisms in $\mathbf{H}(k)$ are called $\mathbb{A}^1$-weak equivalences. For example, if $E \to X$ is an algebraic vector bundle, then it is an $\mathbb{A}^1$-weak equivalence.

As in the topological setting,
the representability theorem for vector bundles holds in $\mathbf{H}(k)$ when restricted to affine varieties. The functor on smooth affine $k$-varieties associating $X$ with $\mathcal{V}_r(X)$ is represented in $\mathbf{H}(k)$ by the infinite Grassmannian $\operatorname{Gr}_r$ \cite[Theorem 5.2.4.8]{Asok2021}:
$$
\mathcal{V}_r(X) \xrightarrow{\sim} \operatorname{Hom}_{\mathbf{H}(k)}(X, \operatorname{Gr}_r).
$$

In particular, if $X$ is a smooth affine variety over $k$ that is $\mathbb{A}^1$-contractible, then every algebraic vector bundle over $X$ is trivial. However, outside the affine case, this fails: there exists an $\mathbb{A}^1$-contractible smooth strictly quasi-affine complex fourfold admitting a nontrivial rank-$2$ vector bundle \cite[Example 5.3.5.4]{Asok2021}. Thus, there exists a topologically contractible complex fourfold admitting a nontrivial algebraic vector bundle; however, this fourfold is not affine.


From the perspective of model categories, there is an $\mathbb{A}^1$-connected component sheaf of $X$, denoted $\pi_0^{\mathbb{A}^1}(X)$. This is the Nisnevich sheaf associated with the presheaf
$$
U \in \mathbf{Sm}/k \mapsto \operatorname{Hom}_{\mathbf{H}(k)}(U, X).
$$
A variety $X$ is called $\mathbb{A}^1$-connected if $\pi_0^{\mathbb{A}^1}(X)$ is isomorphic to the trivial sheaf $\operatorname{Spec} k$. Beyond this abstract definition, understanding the geometry of $\pi_0^{\mathbb{A}^1}(X)$ is crucial and this was initiated by Asok and Morel. They \cite{AsokMorel} (see also \cite[Corollary 3.10]{Balwe2015}) proved that a smooth proper variety $X$ over $k$ is $\mathbb{A}^1$-connected if and only if it is connected in the naive sense. Specifically, for every finitely generated separable field extension $L/k$ and any two points $x, y \in X(L)$, there exists a chain of affine lines $\gamma_1, \dots, \gamma_n: \mathbb{A}^1_L \to X$ (this means $\gamma_i(1) = \gamma_{i+1}(0)$ for all $i$) such that $\gamma_1(0) = x$, $\gamma_n(1) = y$. Thus, if $X$ is an $\mathbb{A}^1$-connected proper variety, then $X$ is rationally connected and also, $\CH_0(X)$ is universally trivial. An $\mathbb{A}^1$-connected smooth variety $X$ (not necessarily proper) is $\mathbb{A}^1$-uniruled, that is $X$ admits a dominant morphism
$$\mathbb{A}^1 \times W \to X, \text{ with } \dim(W) = \dim(X) - 1;$$
therefore $X$ has negative logarithmic Kodaira dimension \cite[Theorem 2.11, Corollary 2.12]{CHOUDHURY2026}. 

The property of being $\mathbb{A}^1$-connected is a birational invariant in the category of smooth proper varieties over an algebraically closed field of characteristic zero \cite[Corollary 2.4.6]{AsokMorel}. Every
smooth proper $\mathbb{A}^1$-connected surface is rational \cite[Corollary 2.4.7]{AsokMorel}. Conversely, a rational smooth proper variety (or more generally, a retract rational smooth proper variety) is $\mathbb{A}^1$-connected. 

Unramified cohomology groups, introduced by Colliot-Thélène and M. Ojanguren \cite{Colliot1989}, are useful birational invariants of smooth proper varieties (Section \ref{unramified section}). If $X$ is a stably rational proper variety, then all the unramified cohomology groups $H^p_{\mathrm{nr}}(k(X)/k, \mu_l^{\otimes q})$ are trivial for $p \geq 1$ \cite[Corollary 4.6]{Schreieder2021UnramifiedCA}. The unramified cohomology groups are related to the Picard group and the Brauer group; the Artin--Mumford threefold \cite{ArtinMumford} has nontrivial unramified cohomology. If $X$ is a topologically contractible variety over $\mathbb{C}$, then $H^i_{\mathrm{nr}}(\mathbb{C}(X), \mu_n^{\otimes l})$ vanishes for every $l,n$ and $i=1,2$ (Lemma \ref{contractible vanishing unramified cohomology}). Moreover, if the compactification $\bar{X}$ is a uniruled threefold, then $H^i_{\mathrm{nr}}(\mathbb{C}(X), \mu_n^{\otimes l})$ vanishes for every $l,n$ and $i \geq 1$ (Corollary \ref{contractible vanishing unramified cohomology threefold}). Asok showed that if $X/k$ is $\mathbb{A}^1$-connected, then all the unramified cohomology groups $H^p_{\mathrm{nr}}(k(X)/k, \mu_l^{\otimes q})$ vanish for any integers $p,q > 0$ \cite[Lemma 4.7]{Asok+2013+39+64}. However, in dimensions $3$ and higher, it is not known whether $\mathbb{A}^1$-connected varieties are rational. Since $\mathbb{A}^1$-connected proper varieties are rationally connected, the main theorem implies the following:
\begin{corollary}(Corollary \ref{main corollary}(iii))
    Let $X$ be a topologically contractible smooth affine complex threefold that admits a smooth projective compactification $\bar{X}$ such that $D=\bar{X} \setminus X$ is a divisor with simple normal crossings. If $X$ is also $\mathbb{A}^1$-connected, then every algebraic vector bundle over $X$ is trivial.
\end{corollary}
In higher dimensions, using \cite[Lemma 4.7]{Asok+2013+39+64} and an analysis of Bloch-Ogus spectral sequence implies that if $X$ is a topologically contractible, smooth variety over $\mathbb{C}$ and also $H^0(X,\mathcal H_{\acute et}^3(\mu_l^{\otimes2}))=0$, then $\CH^2(X)$ is a divisible group (Theorem \ref{connectedcontractible}).

In addition to threefolds, we consider the triviality of algebraic vector bundles over topologically contractible smooth affine complex varieties of dimension $d = 4$. Under the assumption of $\AA^1$-connectedness, the criterion in \cite[Corollary 2.27]{syed2024motiviccohomologycycliccoverings} can be stated as follows:
\begin{theorem}(Corollary \ref{fourfold})
    If $X$ is a topologically contractible and $\AA^1$-connected smooth affine complex variety of dimension $d = 4$, then all algebraic vector bundles on $X$ are trivial if $\CH^2(X)$ and $\CH^3(X)$ are trivial.
\end{theorem}




Based on these investigations, it is natural to ask the following questions (see also Question \ref{general motive question} and Question \ref{rational motive question}):
\begin{question} \label{question on trivial motive}
\begin{enumerate}
    \item Let $X$ be a topologically contractible smooth affine complex threefold that admits a smooth projective compactification $\bar{X}$ such that $D=\bar{X} \setminus X$ is a divisor with simple normal crossings. Assume also that $X$ is either $\mathbb{A}^1$-connected or rationally connected. Is the Voevodsky motive $M(X)$ isomorphic to $\mathbb{Z}$ in $\textbf{DM}_{gm}(\mathbb{C}, \mathbb{Z})$?
    \item Does there exist a topologically contractible and $\mathbb{A}^1$-connected smooth affine variety over $\mathbb{C}$ that is not $\mathbb{A}^1$-contractible?
\end{enumerate}
    \end{question}
\subsection{Outline}
In Section \ref{properties of contractible}, we discuss the basic properties of topologically contractible threefolds, including their Brauer group and unramified cohomology. In Section \ref{Chow group of contractible}, we discuss some results on the Chow groups of topologically contractible varieties and their relations with unramified cohomology via the Bloch--Ogus spectral sequence. In Section \ref{vector bundles on contractible}, we present the main result of this article (Theorem \ref{main theorem}), which gives a conditional answer to the generalized Serre question for threefolds, and we further discuss the implications of Bloch's conjecture and the generalized van de Ven question for the generalized Serre question. In Section \ref{SecKR}, we study rationality of Koras--Russell threefolds and specifically discuss certain Koras--Russell threefolds of the third kind.
In Section \ref{motive computation}, based on specific classifications of compactifications, we investigate topologically contractible threefolds at the level of Voevodsky motives.
\subsection{Acknowledgments}
The authors would first like to thank Aravind Asok for his continuous encouragement, discussions, and advice
throughout this work and for reading several earlier drafts. The authors are also grateful to Burt Totaro for his detailed and invaluable comments on an earlier draft, and to Adrien Dubouloz, Jean Fasel, Amalendu Krishna, Yeqin Liu, Thomas Peternell, Stefan Schreieder and Wanchun Shen for their interest and helpful discussions. The first author would also like to thank BIMSA for organizing the 2023 workshop ``Chow--Witt Rings: Computations and Applications,'' where this problem was first brought to his attention. The second author would like to thank Utsav Choudhury and Aritra Mandal for discussions during this work. The second author would also like to thank the University of Southern California for providing resources during this work.

\section{Properties of topologically contractible varieties}\label{properties of contractible}
Let $X$ be a topologically contractible smooth affine variety over $\mathbb{C}$ of dimension $n$
that admits a smooth projective compactification $\bar{X}$ such that $D=\bar{X} \setminus X$ is a divisor with simple normal crossings. Let $D_1, \dots , D_r$ be the irreducible components of $D$.
Then each $D_i$ is a smooth projective variety of dimension $n-1$, and the components meet transversely; however, $D$ need not be smooth along their intersections. The long exact sequence for the pair $(\bar{X},D)$ contains
$$\dots \to H^i(\bar{X}, D;\mathbb Z) \to H^i(\bar{X};\mathbb Z) \to H^i(D;\mathbb Z) \to H^{i+1}(\bar{X}, D;\mathbb Z) \to \dots.$$
Since $X$ is topologically contractible, Poincar\'e duality gives
$$H^i(\bar{X}, D;\mathbb Z) \cong H^i_c(X;\mathbb Z) \cong H_{2n-i}(X;\mathbb Z)= 0 \quad \text{for }0 \leq i \leq 2n-1.$$
Therefore we have the following immediate topological properties of $\bar{X}$ (see also \cite[Section 3]{MS1990}):
\begin{enumerate}
    \item $H^i(\bar{X}, \mathbb{Z}) \cong H^i(D, \mathbb{Z})$ for $0 \leq i \leq 2n-2$.
    \item The boundary $D$ is connected, and since $X$ is affine, $D$ has pure codimension $1$ in $\bar{X}$.
    \item $H^{2n-1}(\bar{X}, \mathbb{Z}) \cong H_1(\bar{X}, \mathbb{Z}) =0$. This implies that the Hodge numbers
$$h^{1,0}(\bar{X}) = h^{0,1}(\bar{X}) = 0.$$
In fact, $\bar{X}$ is simply connected (Theorem \ref{simply connected}). Also observe that, if $\bar{X}$ is rationally connected, then it is always simply connected.
\item $H^2(\bar{X}, \mathbb{Z})$ is torsion-free. The dimension of the complex vector space $H^2(\bar{X}, \mathbb{C})$ equals the number of irreducible components of $D$; that is, $\dim_\mathbb{C}H^2(\bar{X}, \mathbb{C}) = r$. Thus, $H^2(\bar{X}, \mathbb{Z})$ is isomorphic to $\mathbb Z^{\oplus r}$.
\end{enumerate}
The following lemma is well known. 
\begin{lemma}\label{simply connected}
    Let $Y$ be a
    variety over $\mathbb{C}$ and let $U$ be a Zariski open subset of $Y$. 
    Then the inclusion $U \subset Y$ induces a surjection $\pi_1(U) \to \pi_1(Y)$. In particular, if $X$ and $\bar{X}$ are as above and $X$ is simply connected, then $\bar{X}$ is also simply connected \cite[Lemma 4.19, Chapter 4]{debarre}.
\end{lemma}
The generalized Mumford theorem implies that if $Y$ is a smooth projective variety with $\CH_0(Y) \cong \mathbb{Z}$---in particular, if $Y$ is rationally connected---then $h^{p,0}(Y) = 0$ for every $p \geq 1$ \cite[Theorem 3.13]{voisinannals}. Mixed Hodge theory gives the same vanishing for compactifications of topologically contractible varieties.

\begin{theorem} \label{regular q-forms of contractible}
    Let $X$ be a smooth, topologically contractible complex algebraic variety. Let $\bar{X}$ be a smooth projective compactification of $X$ such that the boundary $D = \bar{X} \setminus X$ is a divisor with simple normal crossings. Then the cohomology of the structure sheaf of $\bar{X}$ vanishes in all positive degrees:
    $$H^i(\bar{X}, \mathcal{O}_{\bar{X}}) = 0 \quad \text{for all } i > 0.$$
\end{theorem}
\begin{proof}
    Let $j:X\hookrightarrow\bar X$ be the open immersion. By Deligne's theory, the singular cohomology groups of both spaces carry natural mixed Hodge structures, and the structure on $H^i(\bar X,\mathbb C)$ is pure. The inclusion induces
    $$j^*: H^i(\bar{X}, \mathbb{C}) \to H^i(X, \mathbb{C}).$$
    This is a morphism of mixed Hodge structures and is strict for the Hodge filtration \cite[Th\'eor\`em 2.3.5]{PMIHES_1971__40__5_0}. The localization sequence gives
    $$\dots \to H^i_D(\bar{X}, \mathbb{C}) \xrightarrow{\gamma} H^i(\bar{X}, \mathbb{C}) \xrightarrow{j^*} H^i(X, \mathbb{C}) \to \dots.$$
    Thus $\ker(j^*)=\operatorname{Im}(\gamma)$. Let $\widetilde D\to D$ be a resolution of singularities. By \cite[Corollary 7.3]{hodge},
    $$\operatorname{Im}\bigl(H^i_D(\bar{X}, \mathbb{C}) \xrightarrow{\gamma} H^i(\bar{X}, \mathbb{C})\bigr)
    =\operatorname{Im}\bigl(H^{i-2}(\widetilde D, \mathbb{C}) \to H^i(\bar{X}, \mathbb{C})\bigr).$$
    The Gysin map has Hodge type $(1,1)$, and therefore
    $$\ker(j^*) = \operatorname{Im}(\gamma) \subset F^1 H^i(\bar{X}, \mathbb{C}).$$
    Strictness now gives an injection on the zeroth graded pieces:
    $$\operatorname{Gr}_F^0(j^*): \operatorname{Gr}_F^0 H^i(\bar{X}, \mathbb{C}) \hookrightarrow \operatorname{Gr}_F^0 H^i(X, \mathbb{C}).$$
    By classical Hodge theory on the smooth projective variety $\bar{X}$,
    $$\operatorname{Gr}_F^0 H^i(\bar{X}, \mathbb{C}) \cong H^{0,i}(\bar{X}) \cong H^i(\bar{X}, \mathcal{O}_{\bar{X}}).$$
    By hypothesis, $X$ is topologically contractible, so $H^i(X,\mathbb C)=0$ for $i>0$. Consequently,
    $$\operatorname{Gr}_F^0 H^i(X, \mathbb{C}) = 0.$$
    The injectivity above therefore forces $H^i(\bar{X}, \mathcal{O}_{\bar{X}})=0$ for all $i>0$, completing the proof.
\end{proof}
The exponential exact sequence and GAGA give the following corollary.
\begin{corollary} \label{compactification picard}
    The first Chern class map 
    $$\Pic(\bar{X}) \to H^2(\bar{X}, \mathbb{Z})$$
    is an isomorphism that maps each class $[D_i]$ to a generator. Thus, $\Pic(\bar{X})$ is isomorphic to $\mathbb Z^{\oplus r}$ and is generated by the line bundles associated with the irreducible components of $D$.
\end{corollary}
In view of Bloch's conjecture, it is natural to ask the following:
\begin{question}
   Let $X$ be a topologically contractible smooth affine variety over $\mathbb{C}$
that admits a smooth projective compactification $\bar{X}$ such that $D=\bar{X} \setminus X$ is a divisor with simple normal crossings. Is $\CH_0(\bar{X}) \cong \mathbb{Z}$ (or equivalently, is $\CH_0(\bar{X})$ supported on a curve, by Corollary \ref{main corollary})? An affirmative answer implies an affirmative answer to the generalized Serre question in dimension $3$; see also Question \ref{version serre question}.
\end{question}
\subsection{Brauer group of a topologically contractible variety}
The classical examples of Artin and Mumford \cite{ArtinMumford} are unirational, and hence rationally connected, smooth proper complex varieties that are not rational and have non-trivial Brauer group. In particular, these varieties are not $\AA^1$-connected; cf. Remark \ref{KRremark}(v). It is natural to ask whether any of them can be a compactification of a topologically contractible affine threefold.

Let $X$ be a smooth variety over $\mathbb{C}$. Its analytification, denoted by $X^{\mathrm{an}}$, has underlying set $X(\mathbb C)$. There is a canonical morphism
$$X^{\mathrm{an}} \to X.$$
If $X$ is affine, then $X^{\mathrm{an}}$ is a Stein space.
\begin{lemma}
Let $X$ be a topologically contractible smooth complex affine variety, and let $E$ be an algebraic, topological, or holomorphic line bundle on $X$. Then $E$ is trivial. In particular,
$$\Pic(X)=H^1(X^{\mathrm{an}},\mathcal O^*)=H^2(X,\mathbb Z)=0.$$
\end{lemma}
\begin{proof}
    The exponential map $\exp: \mathcal{O} \to \mathcal{O}^*$ induces the short exact sequence of analytic sheaves
    $$0 \to \underline{\mathbb{Z}} \to \mathcal{O} \xrightarrow{\exp(2\pi i\,\cdot)} \mathcal{O}^* \to 0,$$
    where $\underline{\mathbb{Z}}$ is the constant sheaf given by $\mathbb{Z}$.
    This induces the long exact sequence
    $$\dots \to H^1(X^{\mathrm{an}}, \underline{\mathbb{Z}}) \to H^1(X^{\mathrm{an}}, \mathcal{O}) \to H^1(X^{\mathrm{an}}, \mathcal{O}^*) \to H^2(X^{\mathrm{an}}, \underline{\mathbb{Z}}) \to H^2(X^{\mathrm{an}}, \mathcal{O}) \to \dots .$$
    By Cartan's theorem for Stein spaces, 
    $$H^1(X^{\mathrm{an}}, \mathcal{O}) \cong H^2(X^{\mathrm{an}}, \mathcal{O}) \cong 0.$$
    Therefore, for a smooth affine complex variety,
    $$H^1(X^{\mathrm{an}}, \mathcal{O}^*) \cong H^2(X^{\mathrm{an}}, \underline{\mathbb{Z}})\cong H^2(X, \mathbb{Z}),$$
    where $H^2(X, \mathbb{Z})$ denotes singular cohomology, and the second isomorphism follows from the comparison between singular cohomology and sheaf cohomology with constant coefficients. Therefore, topological and holomorphic line bundles on a smooth affine complex variety agree up to isomorphism.
    
    If $X$ is topologically contractible, every topological or holomorphic line bundle on $X$ is therefore trivial. A result of Gurjar gives the same conclusion for algebraic line bundles \cite{Gurjar1980}.
    
\end{proof}
\begin{lemma} \label{algebraic analytic injection}
  Let $X$ be a topologically contractible smooth complex affine variety. Then the comparison map $\Phi:X^{\mathrm{an}} \to X$ induces an injection
  $$H^2_{\text{\'et}}(X, \mathbb{G}_m) \to H^2(X^{\mathrm{an}}, \mathcal{O}^*).$$
  Thus the algebraic Brauer group in the \'etale topology injects into the analytic Brauer group.
\end{lemma}
\begin{proof}
    There is a Kummer sequence in the \'etale topology
    $$0 \to \mu_n \to \mathbb{G}_m \xrightarrow{(\cdot)^n} \mathbb{G}_m \to 0$$
    and, similarly, in the complex analytic topology,
    $$0 \to \mu_n \to \mathcal{O}^* \xrightarrow{(\cdot)^n} \mathcal{O}^* \to 0.$$
    The comparison map $X^{\mathrm{an}} \to X$ induces the following commutative diagram between the long exact sequences in cohomology:
   \[
    \begin{tikzcd}[column sep=small]
    \dots H^1_{\text{\'et}}(X, \mathbb{G}_m) \arrow[r] \arrow[d] 
    & H^1_{\text{\'et}}(X, \mathbb{G}_m) \arrow[r] \arrow[d] 
    & H^2_{\text{\'et}}(X, \mu_n) \arrow[r] \arrow[d, "\Phi^*"]
    & H^2_{\text{\'et}}(X, \mathbb{G}_m) \arrow[r] \arrow[d]  
    & H^2_{\text{\'et}}(X, \mathbb{G}_m) \arrow[d] \dots
    \\
     \dots H^1(X^{\mathrm{an}}, \mathcal{O}^*) \arrow[r]
    & H^1(X^{\mathrm{an}}, \mathcal{O}^*) \arrow[r]
    & H^2(X^{\mathrm{an}}, \mu_n) \arrow[r]
    & H^2(X^{\mathrm{an}}, \mathcal{O}^*) \arrow[r]
    & H^2(X^{\mathrm{an}}, \mathcal{O}^*)
    \dots 
    \end{tikzcd}
    \] 
    From the long exact sequences, the images of the maps 
    $$H^2_{\text{\'et}}(X, \mu_n) \to H^2_{\text{\'et}}(X, \mathbb{G}_m) \quad \text{and} \quad H^2(X^{\mathrm{an}}, \mu_n) \to H^2(X^{\mathrm{an}}, \mathcal{O}^*)$$
    are precisely the $n$-torsion subgroups of the respective groups.

    Now suppose that $\alpha \in \Br(X) \cong H^2_{\text{\'et}}(X, \mathbb{G}_m) = H^2_{\text{\'et}}(X, \mathbb{G}_m)_{\mathrm{tors}}$ maps to $0$ in $H^2(X^{\mathrm{an}}, \mathcal{O}^*)$. Then $\alpha$ is an $n$-torsion element of $H^2_{\text{\'et}}(X, \mathbb{G}_m)$ for some $n$. The Kummer sequence induces a surjection
    $$H^2_{\text{\'et}}(X, \mu_n) \to \Br(X)[n],$$
    where $\Br(X)[n]$ denotes the $n$-torsion subgroup of $\Br(X)$. Thus there is some $\beta \in H^2_{\text{\'et}}(X, \mu_n)$ that maps to $\alpha$. Its image $\Phi^*(\beta)$ lies in the kernel of
    $$H^2(X^{\mathrm{an}}, \mu_n) \to H^2(X^{\mathrm{an}}, \mathcal{O}^*).$$
    But $H^1(X^{\mathrm{an}}, \mathcal{O}^*)$ is trivial, since $X$ is contractible. So $\Phi^*(\beta) = 0$. Thus, $\beta = 0$, which follows from the comparison isomorphism
    $$H^2_{\text{\'et}}(X, \mu_n) \xrightarrow{\cong} H^2(X^{\mathrm{an}}, \mu_n).$$
    Hence $\alpha = 0$. This completes the proof.
\end{proof}
\begin{corollary} \label{contractible trivial Brauer group}
    Let $X$ be a topologically contractible smooth complex affine variety. Then the algebraic Brauer group $\Br(X)$ is trivial.
\end{corollary}
\begin{proof}
    Consider the exponential short exact sequence among the complex analytic sheaves
    $$0 \to \underline{\mathbb{Z}} \to \mathcal{O} \xrightarrow{\exp(2\pi i\,\cdot)} \mathcal{O}^* \to 0,$$
    where $\underline{\mathbb{Z}}$ is the constant sheaf given by $\mathbb{Z}$.
    This induces the long exact sequence
    $$\dots \to H^2(X^{\mathrm{an}}, \mathcal{O}) \to H^2(X^{\mathrm{an}}, \mathcal{O}^*) \to H^3(X^{\mathrm{an}}, \underline{\mathbb Z}) \to H^3(X^{\mathrm{an}}, \mathcal{O}) \to \dots.$$
    Since $X$ is a smooth affine variety over $\mathbb{C}$, $X^{\mathrm{an}}$ is Stein. Thus, by Cartan's theorem,
    $$H^i(X^{\mathrm{an}}, \mathcal{O}) =0 \quad \text{for every } i \geq 1.$$
    The comparison between singular cohomology and analytic sheaf cohomology with constant coefficients gives
    $$H^i(X^{\mathrm{an}}, \underline{\mathbb{Z}}) \cong H^i(X(\mathbb{C}), \mathbb{Z}) =0,$$
since $X$ is contractible.
Thus the long exact sequence shows that $H^2(X^{\mathrm{an}}, \mathcal{O}^*)$ is trivial. Lemma \ref{algebraic analytic injection} then implies that the algebraic Brauer group $\Br(X) \cong H^2_{\text{\'et}}(X, \mathbb{G}_m)$ is trivial.

\end{proof}
\begin{theorem} \label{trivial brauer group contractible proper}
    Let $X$ be a topologically contractible affine variety over $\mathbb{C}$ admitting an open immersion $X \hookrightarrow \bar{X}$ into a smooth variety $\bar{X}$. Then $\Br(\bar{X})$ is trivial.
\end{theorem}
\begin{proof}
   By \cite[Theorem 3.5.7]{CT}, the open immersion $X \hookrightarrow \bar{X}$ induces an injective map between the Brauer groups
    $$\Br(\bar{X}) \to \Br(X).$$
    The assertion now follows from Corollary \ref{contractible trivial Brauer group}.
\end{proof}
\begin{corollary}
    Let $Y/S$ be a standard conic bundle over a rational surface $S$ with disconnected discriminant locus $C$, and let $X \subset Y$ be an affine open subvariety. Then $X$ cannot be topologically contractible. In particular, the Artin--Mumford example contains no topologically contractible affine open subvariety \cite{ArtinMumford}.
\end{corollary}
\begin{proof}
    The proof follows from \cite[Theorem 8.3.3]{AG5} and Theorem \ref{trivial brauer group contractible proper}.
\end{proof}
\subsection{Unramified cohomology of a topologically contractible variety} \label{unramified section}
Let $X$ be a topologically contractible variety over $\mathbb{C}$ admitting a smooth projective compactification $\bar{X}$ such that $D=\bar{X} \setminus X$ is a divisor with simple normal crossings. We first recall some definitions and properties of unramified cohomology, introduced by Colliot-Th\'el\`ene and Ojanguren \cite{Colliot1989}. Since the base field is $\mathbb{C}$, the \'etale sheaf $\mu_n$ of $n$-th roots of unity is isomorphic to the constant sheaf $\mathbb{Z}/n\mathbb{Z}$; thus all tensor powers $\mu_n^{\otimes l}$ are non-canonically isomorphic \cite[Section 2]{Schreieder2021UnramifiedCA}.
\begin{definition}
Let $\mathrm{Sm}/\mathbb{C}$ denote the category of smooth $\mathbb{C}$-schemes of finite type. Let $M$ be a strictly $\mathbb{A}^1$-invariant Nisnevich sheaf on $\mathrm{Sm}/\mathbb{C}$. Taking colimits extends $M$ to the category of essentially smooth $\mathbb C$-schemes. An open immersion $U \to X$ induces an injection $M(X) \to M(U)$ \cite[Definition 4.1]{Asok+2013+39+64}.

Let $X \in \mathrm{Sm}/\mathbb{C}$ be irreducible. For $x \in X^{(1)}$, let $v_x$ be the associated discrete valuation of $\mathbb{C}(X)/\mathbb{C}$, where $\mathbb{C}(X)$ denotes the function field of $X$. Then $M(\Spec \mathcal{O}_{v_x})$ is contained in $M(\Spec \mathbb{C}(X))$, since the generic point $\Spec \mathbb{C}(X) \to \Spec \mathcal{O}_{v_x}$ is an open immersion. The map $M(X) \to M(\Spec \mathcal{O}_{v_x})$ is also injective. Define
$$M^{\mathrm{nr}}(X):= \bigcap_{x \in X^{(1)}} M(\Spec \mathcal{O}_{v_x}),$$
where the intersection is taken inside $M(\Spec \mathbb{C}(X))$.
\end{definition}
\begin{theorem}
    For every irreducible $X \in \mathrm{Sm}/\mathbb{C}$, the induced injection $M(X) \to M^{\mathrm{nr}}(X)$ is an isomorphism \cite[Lemma 4.2]{Asok+2013+39+64}.
\end{theorem}
\begin{definition} \label{unramified definition}
    Consider the presheaf on $\mathrm{Sm}/\mathbb{C}$
    $$U \mapsto H^p_{\text{\'et}}(U, \mu_n^{\otimes q})$$
where $\mu_n$ is the \'etale sheaf of $n$-th roots of unity. Let $\mathcal{H}^p_{\text{\'et}}(\mu_n^{\otimes q})$ be the associated Zariski or Nisnevich sheaf, which is strictly $\mathbb{A}^1$-invariant. The group $\HH^0_{\mathrm{Nis}}(X, \mathcal{H}^p_{\text{\'et}}(\mu_n^{\otimes q}))$ consists of the classes unramified on $X$. If $X$ is proper, this group equals the classical unramified cohomology group $\HH^p_{\mathrm{nr}}(\mathbb{C}(X)/\mathbb C, \mu_n^{\otimes q})$ introduced by Colliot-Th\'el\`ene and Ojanguren \cite{Colliot1989}.
\end{definition}

\begin{remark} \label{unramified clarification}
    Let $K = \mathbb{C}(X)$ be the function field of a smooth variety $X$. The group
    $$\HH^0_{\mathrm{Nis}}(X, \mathcal{H}^p_{\text{\'et}}(\mu_n^{\otimes q}))$$
    consists of the classes in $\HH^p_{\text{\'et}}(K, \mu_n^{\otimes q})$ whose residues vanish at the codimension-one points of $X$:
    \begin{align*}
    H^p_{\mathrm{nr}}(X, \mu_n^{\otimes q})
    &\coloneqq \bigcap_{x \in X^{(1)}} \ker(\partial_x),\\
    \partial_x:\ H^p_{\text{\'et}}(K, \mu_n^{\otimes q})
    &\longrightarrow H^{p-1}_{\text{\'et}}(\kappa(x), \mu_n^{\otimes (q-1)}).
    \end{align*}
    By contrast, the classical birational unramified cohomology is defined using all geometric rank-one discrete valuations of $K/\mathbb C$:
    $$ H^p_{\mathrm{nr}}(K/\mathbb{C}, \mu_n^{\otimes q}) \coloneqq \bigcap_v \ker \left( \partial_v: H^p_{\text{\'et}}(K, \mu_n^{\otimes q}) \to H^{p-1}_{\text{\'et}}(\kappa(v), \mu_n^{\otimes (q-1)}) \right).$$
    This is a stable birational invariant of $K$. If $X$ is affine, divisorial valuations at infinity are not visible among the points of $X^{(1)}$. Hence one has only an inclusion
    $$ H^p_{\mathrm{nr}}(K/\mathbb{C}, \mu_n^{\otimes q}) \subseteq \HH^0_{\mathrm{Nis}}(X, \mathcal{H}^p_{\text{\'et}}(\mu_n^{\otimes q})).$$
    A class may therefore be unramified on $X$ but ramified along a divisor at infinity. If $X$ is proper, the two groups agree:
    $$\HH^0_{\mathrm{Nis}}(X, \mathcal{H}^p_{\text{\'et}}(\mu_n^{\otimes q})) = H^p_{\mathrm{nr}}(K/\mathbb{C}, \mu_n^{\otimes q}).$$
\end{remark}

\begin{lemma} \label{contractible vanishing unramified cohomology}
    Let $X$ be a topologically contractible smooth affine variety over $\mathbb{C}$.
    Then the unramified cohomology groups $H^1_{\mathrm{nr}}(\mathbb{C}(X)/\mathbb C, \mu_n^{\otimes l})$ and $H^2_{\mathrm{nr}}(\mathbb{C}(X)/\mathbb C, \mu_n^{\otimes l})$ vanish for every $l$ and $n$.
\end{lemma}
\begin{proof}
Since $X$ is contractible, $\bar{X} \setminus X$ is connected, and the Picard group $\Pic(\bar{X}) \cong H^2(\bar{X}, \mathbb{Z})$ is torsion-free of rank equal to the number of irreducible components of $\bar{X} \setminus X$; see Corollary \ref{compactification picard}.
Thus, by \cite[Proposition 4.2.1]{CTnotes}, the degree-one unramified cohomology groups $H^1_{\mathrm{nr}}(\mathbb{C}(X)/\mathbb C, \mu_n^{\otimes l})$ vanish for every $l$ and $n$.

Theorem \ref{trivial brauer group contractible proper} gives $\Br(\bar{X})=0$. Therefore, by \cite[Proposition 4.2.3]{CTnotes}, the degree-two unramified cohomology groups $H^2_{\mathrm{nr}}(\mathbb{C}(X)/\mathbb C, \mu_n^{\otimes l})$ vanish for every $l$ and $n$.
\end{proof}
\begin{remark}
    If $X \in \mathrm{Sm}/k$ is $\mathbb{A}^1$-connected, but not necessarily topologically contractible, then $\Br(X)$ is trivial \cite[Proposition 4.1.2]{AsokMorel}. There are, however, topologically contractible affine surfaces that are not $\mathbb{A}^1$-connected; the Ramanujam surface is one example \cite{Ramanujam71,ChoudhuryRoy+2024+55+80}. Such surfaces have trivial Brauer group by Corollary \ref{contractible trivial Brauer group}. If $X$ is a topologically contractible surface, then $H^i_{\mathrm{nr}}(\mathbb{C}(X)/\mathbb C, \mu_n^{\otimes l})$ vanishes for every $i>0$ and all $n$ and $l$. Thus it is important to understand the unramified cohomology groups $H^0_{\text{Nis}}(X, \mathcal{H}^i_{\text{\'et}}(\mu_n^{\otimes l}))$ ($\supseteq H^i_{\mathrm{nr}}(\mathbb{C}(X), \mu_n^{\otimes l})$) of such surfaces $X$; see Definition \ref{unramified definition} and Remark \ref{unramified clarification}.
\end{remark}
\begin{corollary} \label{contractible vanishing unramified cohomology threefold}
     Let $X$ be a topologically contractible smooth affine threefold over $\mathbb{C}$ admitting a smooth projective compactification $\bar{X}$ that is uniruled. Then all the unramified cohomology groups $H^i_{\mathrm{nr}}(\mathbb{C}(X)/\mathbb C, \mu_n^{\otimes l})$ vanish for every $i>0$.
\end{corollary}
\begin{proof}
    The groups in degrees $1$ and $2$ vanish by Lemma \ref{contractible vanishing unramified cohomology}. Since $X$ is a threefold, the groups in degrees greater than $3$ vanish by cohomological dimension \cite[Section 2.1]{Schreieder2021UnramifiedCA}. Finally, since $\bar{X}$ is uniruled, the degree-three groups vanish by \cite[Corollaire 6.2]{CTViosin}.
\end{proof}
\begin{remark}\label{biholomorphic}
    If $X$ has dimension $n \geq 3$, then the Dimca--Ramanujam theorem implies that $X$ is diffeomorphic to $\mathbb{C}^n$ \cite[Theorem 3.2]{zaidenberglecture}, but it need not be biholomorphic to $\mathbb{C}^n$. For example, let $S$ be the Ramanujam surface. It is topologically contractible but not simply connected at infinity \cite{Ramanujam71}. Zaidenberg's strong analytic cancellation theorem for smooth varieties of log-general type implies that $S \times \mathbb{C}$ is not biholomorphic to $\mathbb{C}^3$ \cite[Theorem 1.10]{zaidenberg1993}, although it is diffeomorphic to $\mathbb{C}^3$. Every algebraic vector bundle on $S\times_\mathbb C\mathbb A^1_\mathbb C$ is trivial, while $S$ is not $\mathbb A^1$-contractible \cite[Theorem 5.1]{ChoudhuryRoy+2024+55+80}.
    
    If $X$ is a threefold biholomorphic to $\mathbb{C}^3$, then $\bar{X}$ has negative Kodaira dimension \cite[Section 3]{MS1990}. For smooth projective threefolds, three-dimensional minimal model theory identifies negative Kodaira dimension with uniruledness; see \cite[Section 1, Chapter IV]{Kollar}. Thus $\bar{X}$ is uniruled. If $\bar{X}$ is uniruled, an affirmative answer to Bloch's conjecture for surfaces implies an affirmative answer to the generalized Serre question; see Corollary \ref{surface bloch implies serre}. Compactifications of such threefolds have been studied in, for example, \cite{MS1990,Kishimoto2005}.
    It is not known whether a projective compactification of a topologically contractible threefold is always uniruled. This leads to the following question.
    \end{remark}
    \begin{question}
        Let $\bar{X}$ be a projective compactification of a topologically contractible smooth affine threefold $X$ over $\mathbb{C}$. Is $\bar{X}$ uniruled or, equivalently, does $\bar{X}$ have negative Kodaira dimension? Also by \cite[Corollary 0.3]{BDPP}, is $K_{\bar X}$ not pseudo-effective?
    \end{question}
   If $X$ is a topologically contractible affine variety, then $H^i_{\mathrm{nr}}(\mathbb{C}(X)/\mathbb C, \mu_n^{\otimes l})$ vanishes for $i=1,2$ by Lemma \ref{contractible vanishing unramified cohomology}. If, moreover, $X$ is a threefold and $\bar{X}$ is uniruled, then these groups vanish in every positive degree $i>0$ by Corollary \ref{contractible vanishing unramified cohomology threefold}. Their vanishing is an important obstruction in rationality questions.
We end the section with the following question:
\begin{question} \label{van unramified question}

    If $X$ is a topologically contractible smooth affine threefold over $\mathbb{C}$, does $H^i_{\mathrm{nr}}(\mathbb{C}(X)/\mathbb C, \mu_n^{\otimes l})$ vanish for every $i>0$?
\end{question}

\section{Chow groups of topologically contractible varieties}\label{Chow group of contractible}
In this section, we discuss some results on the Chow groups of topologically contractible affine varieties.
Let $X$ be a topologically contractible variety over $\mathbb{C}$ admitting a smooth projective compactification $\bar{X}$ such that $D=\bar{X} \setminus X$ is a divisor with simple normal crossings. If $\bar{X}$ is rationally connected, then $\CH_0(\bar{X})\cong\mathbb{Z}$ via the degree map
$$\operatorname{deg}: \CH_0(\bar{X}) \to \mathbb{Z}.$$
Therefore, the proper pushforward
$$\CH_0(D) \to \CH_0(\bar{X}),$$
induced by the inclusion $D \subset \bar{X}$, is surjective. The localization sequence for Chow groups then gives $\CH_0(X)=0$. If $X$ is $\mathbb{A}^1$-connected, then its projective compactification
$\bar{X}$ is also $\mathbb{A}^1$-connected by \cite[Proposition 5.4.2.7]{Asok2021}. Thus $\bar{X}$ is rationally connected \cite[Theorem 2.4.3 and Corollary 2.4.4]{AsokMorel}; see also \cite[Corollary 3.10]{Balwe2015}.

\begin{lemma} \label{highest chow group trivial}
    Suppose that $X$ is $\mathbb{A}^1$-connected. Then $\CH_0(\bar{X})\cong\mathbb{Z}$, and hence $\CH_0(X)=0$.
\end{lemma}
\begin{remark}\label{zerocycleremark}
\begin{enumerate}
    \item By the proof of \cite[Proposition 5.4.2.7]{Asok2021}, Lemma \ref{highest chow group trivial} also holds under the assumption that the zeroth $\mathbb{A}^1$-homology sheaf $\mathbf{H}_0^{\mathbb{A}^1}(X)$ is isomorphic to the constant sheaf $\mathbb{Z}$.
    \item If $X \in \mathrm{Sm}/k$ is $\mathbb{A}^1$-connected, then $\mathbf{H}_0^{\mathbb{A}^1}(X)\cong\mathbb{Z}$ \cite[Theorem 1]{Asok+2013+39+64}; the converse holds when $X$ is proper \cite[Theorem 5]{Asok+2013+39+64}. For a general smooth variety, it is not known whether these conditions are equivalent.
    \item If $X$ is a smooth affine variety of dimension $d \geq 3$, then 
    $\CH_0(X)$ is uniquely divisible \cite{SRINIVAS1989428}.
    \end{enumerate}
\end{remark}

Bloch and Ogus construct the coniveau spectral sequence \cite{ASENS_1974_4_7_2_181_0}
$$E^{p,q}_2=\HH^p(X,\cH_{\acute{e}t}^q(\mu_l^{\otimes n}))\Rightarrow H_{\acute{e}t}^{p+q}(X, \mu_l^{\otimes n}).$$
We use it to obtain the following statement.

\begin{theorem}\label{connectedcontractible}
    Let $X$ be a topologically contractible complex
    smooth affine variety. Assume that, for every integer $l>1$,
    $$\HH^0(X,\cH_{\acute{e}t}^3(\mu_l^{\otimes 2})))=0.$$
    Then $\CH^2(X)$ is divisible.
\end{theorem}
\begin{proof}
  Fix an integer $l>1$. In the Bloch--Ogus spectral sequence \cite[(4.2)]{ASENS_1974_4_7_2_181_0}, $E^{p,q}_2=0$ for $p>q$. Hence there is an exact sequence
    \begin{equation*}
        0\to E_{\infty}^{0,3}\to E_{2}^{0,3}\to E_{2}^{2,2}\to E_{\infty}^{2,2}\to 0
    \end{equation*}
    and therefore an exact sequence
    \begin{equation*}
        H_{\acute{e}t}^3(X, \mu_l^{\otimes 2})\to \HH^0(X,\cH_{\acute{e}t}^3(\mu_l^{\otimes 2})))\to \HH^2(X,\cH_{\acute{e}t}^2(\mu_l^{\otimes 2})))\to H_{\acute{e}t}^4(X, \mu_l^{\otimes 2}).
    \end{equation*}
     Since $X$ is topologically contractible, the comparison theorem between \'etale and singular cohomology makes the two outer terms vanish \cite[Theorem 21]{milne}. Thus,
     $$\HH^0(X,\cH_{\acute{e}t}^3(\mu_l^{\otimes 2})))\cong \HH^2(X,\cH_{\acute{e}t}^2(\mu_l^{\otimes 2}))).$$
     The hypothesis therefore implies that $\HH^2(X,\cH_{\acute{e}t}^2(\mu_l^{\otimes 2})))=0$. The norm-residue isomorphism theorem then gives
     $$\CH^2(X)/l\cong\HH^2(X,\cH_{\acute{e}t}^2(\mu_l^{\otimes 2}))=0$$
     \cite{Voevodsky2003,Voevodsky2008OnMC}.
\end{proof}

\begin{remark}
    The proof of Theorem \ref{connectedcontractible} shows the following weaker consequence. If $X$ is a smooth affine complex threefold that is acyclic and whose zeroth $\mathbb{A}^1$-homology sheaf $\mathbf{H}_0^{\mathbb{A}^1}(X)$ is isomorphic to the constant sheaf $\mathbb{Z}$, then $\CH^2(X)$ is divisible.
\end{remark}
Since $\mathbb{A}^1$-connectedness implies $\HH^0(X,\cH_{\acute{e}t}^3(\mu_l^{\otimes 2}))=0$ \cite[Lemma 4.7]{Asok+2013+39+64}, the following is immediate.
\begin{corollary}
    If $X$ is a topologically contractible smooth affine variety over $\CC$ that is also $\mathbb{A}^1$-connected,
    then $\CH^2(X)$ is divisible.
\end{corollary}
\begin{remark}
      If an affine variety is rational, the group $\HH^0_{\mathrm{Nis}}(X, \mathcal{H}^p_{\acute{e}t}(\mu_n^{\otimes q}))$ need not be trivial. Rationality, or even stable rationality, guarantees the vanishing of
    $\HH^0_{\mathrm{Nis}}(\bar{X}, \mathcal{H}^p_{\acute{e}t}(\mu_n^{\otimes q}))$ for a smooth proper model $\bar X$ and every $p,q>0$ \cite[Corollary 4.6]{Schreieder2021UnramifiedCA}; see also Remark \ref{unramified clarification}.
    By contrast, $\mathbb{A}^1$-connectedness implies the vanishing of these groups even when the variety is not proper \cite[Lemma 4.7]{Asok+2013+39+64}.
    For example, consider the punctured affine line $X = \mathbb{A}^1_\mathbb{C} \setminus \{0\}$. For $p=q=1$, the local-to-global spectral sequence gives
    $$\HH^0_{\text{Nis}}(X, \mathcal{H}^1_{\acute{e}t}(\mu_n)) \cong \HH^1_{\acute{e}t}(X, \mu_n).$$
    The Kummer exact sequence identifies the right-hand side with
    $$\mathcal{O}(X)^\times/(\mathcal{O}(X)^\times)^n,$$
    because $\Pic(X)=0$. Since $\mathcal{O}(X)^\times\cong\mathbb C^\times\times\mathbb Z$, generated by constants and the coordinate $t$, this quotient is $\mathbb Z/n\mathbb Z$. Thus the class of $t$ is non-trivial. It is unramified on $X$ but has a residue at the missing origin.
\end{remark}

\begin{remark}
\begin{enumerate}
    \item The simple connectedness of $X$ and the mixed Hodge structure on a resolution $\widetilde D\to D$ give $H^1(D_i,\mathbb Q)=0$ for every irreducible component $D_i$ of $D$. This raises the question whether every boundary component of a compactification $(\bar X,D=\bar X\setminus X)$ of a contractible threefold is rational. The classifications of Fujita, M\"uller, and Kishimoto in the one- and two-component cases imply that each $D_i$ occurring there is a rational surface.
    \item The Danielewski surface $X_n = \{x^n z = y^2 - 1\} \subset \mathbb{A}^3$ ($n\geq 2$) is a rational smooth affine surface with rational boundaries \cite[Section 5.3.4]{Asok2021}, but it is not $\AA^1$-contractible.
\end{enumerate}
   
\end{remark}

Conversely, as in \cite[Theorem 1]{Asok_2013}, consider the group
$$\HH^0(X,\cH_{\acute{e}t}^q(\mu_l^{\otimes n})).$$
Its nonvanishing is an obstruction to $\AA^1$-connectedness. For a topologically contractible smooth complex variety, $\CH^2(X)$ provides such an obstruction.

\begin{theorem}\label{converse}
    Let $X$ be a topologically contractible smooth quasi-projective complex variety. If $\CH^2(X)$ is not torsion and $\CH^2(X)/l$ is non-trivial for some $l$, then $X$ is not $\AA^1$-connected.
\end{theorem}
\begin{proof}
    As in the proof of Theorem \ref{connectedcontractible}, we have
    $$\HH^0(X,\cH_{\acute{e}t}^3(\mu_l^{\otimes 2}))\cong \HH^2(X,\cH_{\acute{e}t}^2(\mu_l^{\otimes 2}))\cong\CH^2(X)/l.$$
    If $\CH^2(X)$ is torsion, then by \cite[Theorem 17]{KUMAR1986483}
    due to Srinivas, $\CH^2(X)$ is trivial. If $\CH^2(X)$ is not torsion and $\CH^2(X)/l$ is non-trivial for some $l$, then $\HH^0(X,\cH_{\acute{e}t}^3(\mu_l^{\otimes2}))$ is non-trivial. The conclusion follows from \cite[Lemma 4.7]{Asok+2013+39+64}.
\end{proof}

\section{Algebraic vector bundles over contractible threefolds}\label{vector bundles on contractible}
Algebraic vector bundles on smooth affine threefolds over $\mathbb{C}$ are classified by their Chern classes in the Chow groups. Results of Mohan Kumar--Murthy and Asok--Fasel imply that, for a threefold $X$, all vector bundles on $X$ are trivial if and only if $\CH^i(X)$ vanishes for every $i>0$ \cite[Theorem 2.1(iii)]{MohanKumarMurthy}; see also \cite[Theorem 1]{Asok12a}. Gurjar and Shastri proved that every topologically contractible surface is rational \cite{GSa,GSb}. Thus, in dimensions at most $2$, every algebraic vector bundle on a topologically contractible smooth complex variety is trivial.

In this section, we study the Chow groups of a topologically contractible smooth affine threefold $X$ over $\mathbb{C}$ in relation to the generalized Serre question (Question \ref{serre question}). Theorem \ref{main theorem} proves that if $X$ admits a smooth projective rationally connected compactification, then $\CH_1(X)$ is trivial; the classification above then gives the triviality of every vector bundle on $X$ (Corollary \ref{main corollary}). Consequently, an affirmative answer to Bloch's conjecture for threefolds implies an affirmative answer to the generalized Serre question in dimension $3$. If $X$ admits an uniruled projective compactification $\bar X$, the corresponding implication can instead be related to Bloch's conjecture for surfaces; see Section \ref{relationtobloch}.
Before proving the main theorem, we recall the following result of Srinivas.
\begin{theorem}\label{srinivas}
 Suppose that $X$ is a smooth variety over $\mathbb{C}$ such that, for some $N>0$, $\CH^2(X)$ is $N$-torsion and $H^3(X, \mathbb{Q})=0$. Then the cycle class map
 $$\operatorname{cl}:\CH^2(X)\longrightarrow H^4(X,\mathbb Z)$$
 is injective \cite[Appendix, Theorem]{KUMAR1986483}.
\end{theorem}
The main result in this section is the following:
\begin{theorem} \label{main theorem}
    Let $X$ be a topologically contractible smooth affine complex threefold admitting a smooth projective compactification $\bar{X}$ such that $D=\bar{X} \setminus X$ is a divisor with simple normal crossings. Assume that $\CH_0(\bar{X})$ is supported on a curve. Then $\CH_1(X)$
    is trivial.
\end{theorem}
    \begin{proof}
    We prove the theorem in two steps.
    \begin{enumerate}
        \item \textbf{Claim 1:} $\CH_1(X)$ is torsion, or equivalently $\CH_1(X)_\mathbb{Q} = 0$, where $G_\mathbb{Q}:= G \otimes_\mathbb{Z} \mathbb{Q}$ for an abelian group $G$.
        \item \textbf{Claim 2:} $\CH_1(X)$ is finitely generated.
    \end{enumerate}
    Together, the two claims imply that $\CH_1(X)$ is annihilated by some positive integer $N$.
    Theorem \ref{srinivas} then implies that $\CH_1(X)$ is trivial.

    We now prove the two claims.
    
    \textbf{Proof of Claim 1.} The localization sequence for Chow groups remains exact after tensoring with $\mathbb{Q}$:
    $$\CH_1(D)_\mathbb{Q} \longrightarrow \CH_1(\bar{X})_\mathbb{Q} \longrightarrow \CH_1(X)_\mathbb{Q} \longrightarrow 0.$$

    It is therefore enough to prove that the proper pushforward of $1$-cycles
    $$j_*: \CH_1(D)_\mathbb{Q} \to \CH_1(\bar{X})_\mathbb{Q},$$
    induced by the closed immersion $j:D\to \bar{X}$ is surjective. Let $\widetilde D$ be a smooth resolution of $D$; it may be disconnected. Write
    $$\sigma: \widetilde{D} \to D$$
    for the proper birational morphism.
    This gives the following composition of morphisms
    $$\CH_1(\widetilde{D})_\mathbb{Q} \xrightarrow{\sigma_*} \CH_1(D)_\mathbb{Q} \xrightarrow{j_*} \CH_1(\bar X)_\mathbb{Q}.$$
    Thus it is enough to prove that $j_* \circ \sigma_* = (j \circ \sigma)_*$ is surjective. We shall prove that
    $$(j\circ \sigma)_*: \CH_1(\widetilde{D})_\mathbb{Q} \to \CH_1(\bar{X})_\mathbb{Q}$$
    is surjective.

    Consider the following commutative diagram
    \[
\xymatrix{
\CH_1(\widetilde{D})_\mathbb{Q} \ar[r]\ar[d] & \CH_1(\bar{X})_\mathbb{Q} \ar[d] \\
\operatorname{Hdg}^2(\widetilde{D}, \mathbb{Q}) \ar[r] & \operatorname{Hdg}^4(\bar{X}, \mathbb{Q}).
}
\]
In the diagram, for a smooth projective variety $Y$ of dimension $n$, the group
\begin{align*}
\operatorname{Hdg}^{2k}(Y, \mathbb{Q})
&:= H^{2k}(Y, \mathbb{Q}) \cap H^{k,k}(Y) \\
&= \{x \in H^{2k}(Y,\mathbb{Q}) \mid \text{the image of } x \text{ in } H^{2k}(Y, \mathbb{C}) \text{ is of type }(k,k)\}.
\end{align*}
This denotes the subgroup of rational Hodge classes in $H^{2k}(Y,\mathbb Q)$.
In the diagram, the vertical maps are the cycle class maps
$$\CH_1(\widetilde{D})_\mathbb{Q} \to H_2(\widetilde{D}, \mathbb{Q}) \quad \text{and} \quad \CH_1(\bar{X})_\mathbb{Q} \to H_2(\bar{X}, \mathbb{Q}),$$
whose images lie in the corresponding groups of rational Hodge classes. The morphism $j \circ \sigma: \widetilde{D} \to \bar{X}$ induces
$$H^2(\widetilde{D},\mathbb Q) \to H^4(\bar{X},\mathbb Q)$$
by Poincar\'e duality. This map restricts to the morphism of rational Hodge classes
$$\operatorname{Hdg}^2(\widetilde{D}, \mathbb{Q}) \to \operatorname{Hdg}^4(\bar{X}, \mathbb{Q}).$$
This gives the commutative diagram above.

For a smooth projective variety of dimension $n$, the rational Hodge conjecture holds in degrees $2$ and $2n-2$. The degree-two case is the Lefschetz theorem on $(1,1)$-classes, and the degree-$(2n-2)$ case follows from it and the hard Lefschetz theorem \cite[Section 2.2.2]{voisinannals}. Therefore, both vertical maps are surjective.

We claim that the bottom horizontal map is also surjective. Since $X$ is contractible, the cohomology long exact sequence gives the maps
$$H^i(\bar{X}, \mathbb{Z}) \to H^i(D, \mathbb{Z})$$
are isomorphisms for $1 \leq i \leq 5$ \cite[Section 3]{MS1990}. The mixed Hodge structure on $\widetilde D$ and \cite[Lemma 7.2]{hodge} then show that
$$H^i(\bar{X}, \mathbb{Z}) \to H^i(\widetilde D, \mathbb{Z})$$
is injective for $1 \leq i \leq 5$. Since $\mathbb Q$ is flat over $\mathbb Z$, the maps
$$H^i(\bar{X}, \mathbb{Q}) \to H^i(\widetilde D, \mathbb{Q})$$
are also injective in this range. These groups are finite-dimensional $\mathbb Q$-vector spaces.

Duality therefore shows that
$$H_i(\widetilde D, \mathbb{Q}) \to H_i(\bar{X}, \mathbb{Q})$$
is surjective for $1\leq i \leq 5$. In particular, Poincar\'e duality gives a surjection
$$H^2(\widetilde D,\mathbb Q) \to H^4(\bar{X},\mathbb Q).$$
We claim that it remains surjective on rational Hodge classes. Since $\bar{X}$ is smooth and projective, its rational Hodge structures are polarizable \cite[Section 2.2.1]{voisinannals}. The surjective morphism of Hodge structures of type $(1,1)$ \cite[Definition 7.22]{voisin1}
$$H^2(\widetilde D,\mathbb Q) \to H^4(\bar{X},\mathbb Q)$$
admits a section
$$s: H^4(\bar{X},\mathbb Q) \to H^2(\widetilde D,\mathbb Q)$$
of type $(-1,-1)$ \cite[Corollary 2.24]{voisinannals}. Hence, for $\alpha \in \operatorname{Hdg}^4(\bar{X}, \mathbb{Q})$, the class $s(\alpha)$ belongs to $\operatorname{Hdg}^2(\widetilde D,\mathbb Q)$. Thus the map restricts to a surjection
$$\operatorname{Hdg}^2(\widetilde D, \mathbb{Q}) \to \operatorname{Hdg}^4(\bar{X}, \mathbb{Q}).$$



To prove that the upper horizontal map $\CH_1(\widetilde D)_{\mathbb{Q}} \to \CH_1(\bar{X})_{\mathbb{Q}}$ is surjective, let $[Z] \in \CH_1(\bar{X})_{\mathbb{Q}}$. The surjectivity of the bottom and left arrows gives a class $[Z_{\widetilde D}] \in \CH_1(\widetilde D)_{\mathbb{Q}}$ such that
$$[(j \circ \sigma)_*(Z_{\widetilde D})] = [Z] \in H_2(\bar{X}, \mathbb{Q}).$$
Consequently,
$$[(j \circ \sigma)_*(Z_{\widetilde D})-Z] \in \CH_1(\bar{X})_{\mathrm{hom}} \otimes_\mathbb{Z} \mathbb{Q}.$$
The map
$$\CH_1(\widetilde D)_{\mathrm{hom}} \otimes_\mathbb{Z} \mathbb{Q} \to \CH_1(\bar{X})_{\mathrm{hom}} \otimes_\mathbb{Z} \mathbb{Q}$$
is surjective by Lemma \ref{homologically trivial}. Hence there is a class $[Z^\prime_{\widetilde D}] \in \CH_1(\widetilde D)_{\mathrm{hom}} \otimes_\mathbb{Z} \mathbb{Q}$ such that
$$[(j \circ \sigma)_*(Z_{\widetilde D})-Z] = [(j \circ \sigma)_*(Z^\prime_{\widetilde D})] \in \CH_1(\bar{X})_{\mathrm{hom}} \otimes_\mathbb{Z} \mathbb{Q}.$$
Thus,
$$[(j\circ\sigma)_*(Z_{\widetilde D}- Z^\prime_{\widetilde D})] = [Z] \in \CH_1(\bar{X})_{\mathbb{Q}}.$$
Thus $\CH_1(\widetilde D)_\mathbb Q\to\CH_1(\bar X)_\mathbb Q$ is surjective. Consequently, $\CH_1(X)_\mathbb{Q}=0$, proving Claim 1 modulo Lemma \ref{homologically trivial}.

We next prove that $\CH_1(X)$ is finitely generated.

\textbf{Proof of Claim 2.} The localization sequence
$$\CH_1(D) \xrightarrow{j_*} \CH_1(\bar{X}) \longrightarrow \CH_1(X) \longrightarrow 0,$$
identifies $\CH_1(X)$ with $\CH_1(\bar{X})/\operatorname{Im}(j_*)$. Since
$$(j \circ\sigma)_*: \CH_1(\widetilde D)_{\mathrm{hom}} \to \CH_1(\bar{X})_{\mathrm{hom}}$$
is surjective by Lemma \ref{homologically trivial},
$$\CH_1(\bar{X})_{\mathrm{hom}} \subseteq \operatorname{Im}(j_*).$$
The integral cohomology groups of the smooth projective threefold $\bar X$ are finitely generated.
Thus from the cycle class map
$$\CH_1(\bar{X}) \to H_2(\bar{X}, \mathbb{Z})\cong H^4(\bar{X}, \mathbb{Z}),$$
we conclude that 
$\CH_1(\bar{X})/\CH_1(\bar{X})_{\mathrm{hom}}$ is finitely generated. Since
$$\CH_1(\bar{X})_{\mathrm{hom}} \subseteq \operatorname{Im}(j_*) \quad \text{and} \quad \CH_1(X) \cong \CH_1(\bar{X})/\operatorname{Im}(j_*),$$
the group $\CH_1(X)$ is finitely generated.

This proves Claim 2. It remains to prove Lemma \ref{homologically trivial}.


\end{proof}
\begin{lemma} \label{homologically trivial}
  Let $\bar{X}$ and $\widetilde D$ be as in the proof of Theorem \ref{main theorem}. Then
$$(j \circ\sigma)_*: \CH_1(\widetilde D)_{\mathrm{hom}} \to \CH_1(\bar{X})_{\mathrm{hom}}$$
is surjective.
\end{lemma}
\begin{proof}
    Consider the Abel--Jacobi maps
$$\Phi^1_{\widetilde D}: \CH^1(\widetilde D)_{\mathrm{hom}} \to J^1(\widetilde D)$$
and
$$\Phi^2_{\bar{X}}: \CH^2(\bar{X})_{\mathrm{hom}} \to J^3(\bar{X}),$$
where $J^1(\widetilde D)$ and $J^3(\bar{X})$ are the intermediate Jacobians associated with $\widetilde D$ and $\bar{X}$, respectively. For a smooth projective variety $Y$, one has \cite[Section 12.1]{voisin1}
$$J^{2k-1}(Y):=H^{2k-1}(Y,\mathbb C)/\bigl(F^kH^{2k-1}(Y,\mathbb C)+H^{2k-1}(Y,\mathbb Z)\bigr),$$
together with an Abel--Jacobi map
$$\CH^k(Y)_{\mathrm{hom}} \to J^{2k-1}(Y),$$
where $\CH^k(Y)_{\mathrm{hom}}$ is the kernel of the cycle class map
$$\CH^k(Y) \to H^{2k}(Y, \mathbb{Z}).$$
For $k=1$, from the exponential exact sequence there is an isomorphism
$$\Pic^0(Y):= \ker\bigl(c_1:\Pic(Y) \to H^2(Y, \mathbb{Z})\bigr) \cong J^1(Y).$$
Thus the Abel-Jacobi map
$$\CH^1(Y)_{\mathrm{hom}} \to J^{1}(Y)$$
is an isomorphism \cite[Section 12.1.3]{voisin1}.
Since $\CH_0(\bar{X})$ is supported on a curve, the Abel--Jacobi map for $k=2$,
$$\CH^2(\bar{X})_{\mathrm{hom}} \to J^{3}(\bar{X}),$$
is an isomorphism \cite[Theorem 6.24 and Remark 6.25]{voisinannals}; see also \cite[Theorem 0.3]{BlochSrinivas,Murre85,Voisin13}.
Now consider the following commutative diagram
 \[
\xymatrix{
\CH^1(\widetilde D)_{\mathrm{hom}} \ar[r]\ar[d]_{AJ_{\widetilde D}} & \CH^2(\bar{X})_{\mathrm{hom}} \ar[d]^{AJ_{\bar{X}}} \\
J^1(\widetilde D) \ar[r] & J^3(\bar{X})
}
\]
The horizontal morphisms are induced by proper pushforward of $1$-cycles, and both vertical maps are
isomorphisms. To show the map
$$J^1(\widetilde D) \to J^3(\bar{X})$$
is surjective, consider the following commutative diagram
 \[
\xymatrix{
H^1(\widetilde D, \mathbb{C}) \ar[r]\ar[d] & H^3(\bar{X}, \mathbb{C}) \ar[d] \\
J^1(\widetilde D) \ar[r] & J^3(\bar{X}).
}
\]
The vertical maps are the canonical quotient maps. By Poincar\'e duality, the map
$$H^1(\widetilde D, \mathbb{C}) \to H^3(\bar{X}, \mathbb{C})$$
is dual to
$$H_3(\widetilde D, \mathbb{C}) \to H_3(\bar{X}, \mathbb{C}),$$
which is surjective by the proof of Claim 1 and \cite[Lemma 7.2]{hodge}. The diagram therefore shows that
$$J^1(\widetilde D) \to J^3(\bar{X})$$
is surjective. Consequently,
$$(j \circ\sigma)_*: \CH_1(\widetilde D)_{\mathrm{hom}} \to \CH_1(\bar{X})_{\mathrm{hom}}$$
is surjective. This completes the proof.
\end{proof}

The classification results of Mohan Kumar--Murthy and Asok--Fasel now give the following corollary.
\begin{corollary} \label{main corollary}
    Let $X$ be a topologically contractible smooth affine complex threefold admitting a smooth projective compactification $\bar{X}$ such that $D=\bar{X} \setminus X$ is a divisor with simple normal crossings. Assume that one of the following conditions holds:
    \begin{enumerate}
    \item $\CH_0(\bar{X})$ is supported on a curve.
    \item $\CH_0(\bar{X}) \cong \mathbb{Z}$.
      \item $\bar{X}$ is rationally connected; for example, $\bar{X}$ is unirational, Fano, or $\mathbb{A}^1$-connected.
 \end{enumerate}       
        Then every algebraic vector bundle on $X$ is trivial. In fact, (1) and (2) are equivalent under the hypotheses above.
         \end{corollary}
\begin{proof}
    We only need to show that if $\CH_0(\bar{X})$ is supported on a curve, then $\CH_0(\bar{X}) \cong \mathbb{Z}$. By Theorem \ref{regular q-forms of contractible}, $h^{0,1}(\bar X)=0$, so the irregularity of the compactification is $q(\bar{X}) = 0$. Thus the Albanese variety $\operatorname{Alb}(\bar{X})$ is trivial. The Albanese kernel
    $$T(\bar{X})=\ker\bigl(\CH_0(\bar{X})_{\deg 0}\to \operatorname{Alb}(\bar{X})\bigr)$$
    is a torsion group by the Bloch--Srinivas decomposition-of-the-diagonal argument; see \cite{BlochSrinivas,Murre85}. On the other hand, Roitman's theorem states that the Albanese map induces an isomorphism on torsion subgroups \cite{R}. Hence its kernel $T(\bar X)$ is torsion-free. It follows that $T(\bar X)=0$. Since $\operatorname{Alb}(\bar X)=0$, we obtain $\CH_0(\bar X)_{\deg 0}=0$, and therefore $\CH_0(\bar X)\cong\mathbb Z$.
\end{proof}
\begin{corollary}\label{contractible rational vector bundles}
    Every algebraic vector bundle on a topologically contractible smooth rational affine complex threefold is trivial.
\end{corollary}

\begin{remark}\label{relation with ABH}
    The conjecture of Asok, Fasel and Hopkins states that every topological vector bundle on a smooth complex cellular variety admits a unique motivic lift \cite[Conjecture 5]{AsokFaselHopkins2020}. For a smooth affine variety, motivic vector bundles are represented by algebraic vector bundles. \cite[Theorem 4]{AsokBachmannHopkins2026} proves the corresponding algebraic--topological comparison under the precise hypothesis of even cellularity. This comparison and Corollary \ref{main corollary} are complementary: the former applies in every dimension and requires no compactification, but assumes even cellularity; the latter does not assume even cellularity, but is specific to threefolds and imposes a condition on the Chow group of a smooth projective compactification.

    There is also a direct connection with the motive computations below. Under the R\"ondigs--{\O}stv{\ae}r equivalence between Voevodsky motives and motivic $H\mathbb Z$-modules \cite{RondigsOstvaer2008}, an isomorphism $M(X)\cong\mathbb Z$ gives $H\mathbb Z[X]\cong H\mathbb Z$; hence $X$ is even cellular in the sense of \cite[Definition 1.4.7]{AsokBachmannHopkins2026}. For surfaces, Asok proved that every topologically contractible smooth complex surface has trivial integral motive \cite[Theorem 1]{Asok_motive}, and hence is even cellular. In dimension $3$, the analogous motivic statement remains open even for rational topologically contractible smooth affine varieties; see Question \ref{rational motive question}. The same observation applies to Koras--Russell threefolds of the first and second kinds, whose motives are trivial \cite{HKO}. For Koras--Russell threefolds of the third kind, however, the motive is not known to be trivial, and no even-cellularity result is presently available; Corollary \ref{korasrussellthird} proves the triviality of their algebraic vector bundles without this hypothesis. 
\end{remark}

By \cite[Corollary 2.6]{kaliman2002bfcactionscontractiblethreefolds}, a smooth topologically contractible affine threefold over $\mathbb{C}$ with a non-trivial algebraic $\CC^+$-action is rational. Hence:
\begin{corollary}\label{Cplusaction}
    Let $X$ be a smooth topologically contractible affine threefold over $\mathbb{C}$ with a non-trivial algebraic $\CC^+$-action. Then every algebraic vector bundle on $X$ is trivial.
\end{corollary}

Recall that the generalized van de Van question asks that
whether a smooth, topologically contractible, complex variety is rational, so we have the following corollary:

\begin{corollary}\label{vdVimpliesSerre}
    Suppose, $X$ is a topologically contractible, smooth, affine, complex threefold that admits a smooth projective compactification $\bar{X}$ such that $D=\bar{X} \setminus X$ is a divisor with simple normal crossings. Then in case of affine threefolds, the positive answer to the generalized van de Ven question implies an affirmative answer to the generalized Serre question (\cite[Question 6 and Question 8]{Asok2021}).
\end{corollary}
    \begin{remark}
         If $X$ is $\mathbb{A}^1$-connected, then the compactification $\bar{X}$ is always $\mathbb{A}^1$-connected \cite[Proposition 5.4.2.7]{Asok2021}. However, it is not true that $X$ is topologically contractible and $\bar{X}$ is $\mathbb{A}^1$-connected, that would imply $X$ is also $\mathbb{A}^1$-connected; for example we can take $X = R \times \mathbb{A}^1$, where $R$ is the Ramanujam surface \cite{Ramanujam71}. Then the compactification $\bar{X}$ is $\mathbb{A}^1$-connected, since any proper, rational variety is $\mathbb{A}^1$-connected \cite{AsokMorel}; but $X$ is not $\mathbb{A}^1$-connected \cite[Theorem 5.1]{ChoudhuryRoy+2024+55+80}.
    \end{remark}
   

In the end if we assume Bloch's conjecture \ref{bloch} for threefolds, then

\begin{corollary}\label{BlochimpliesSerre}
    Suppose, $X$ is a topologically contractible, smooth, affine, complex threefold that admits a smooth projective compactification $\bar{X}$ such that $D=\bar{X} \setminus X$ is a divisor with simple normal crossings. Then an affirmative solution to Bloch's Conjecture in case of $\bar{X}$ implies an affirmative answer to the generalized Serre question. 
\end{corollary}

\subsection{Relation to Bloch's conjecture for surfaces}\label{relationtobloch}
Suppose that $X$ is a topologically contractible affine threefold satisfying the hypotheses of Corollary \ref{main corollary}.  Moreover, assume that $\bar{X}$ is uniruled. By a theorem of Miyaoka-Mori, this is equivalent to the fact that $\bar{X}$ has negative Kodaira dimension \cite[Section1, Chapter IV]{Kollar}.

There exists maximal rationally connected (MRC fibration \cite[Section 5, Chapter IV]{Kollar}), $\pi:\bar{X} \to Z$. Since we have assumed $\bar{X}$ is uniruled, so the dimension of $Z$ is atmost $2$.
After choosing suitable smooth projective birational models, we may assume that $\pi$ is a morphism and
$Z$ is a smooth proper
variety. Moreover,
\begin{lemma}\label{mrc fibration invariant}
\begin{enumerate}
    \item The MRC fibration $\pi$ is Chow constant. The pushforward
$$\pi_*: \CH_0(\bar{X}) \to \CH_0(Z)$$
is an isomorphism \cite[Corollary 2.7]{stapleton}.
\item The MRC fibration $\pi$ is cohomologically constant in the sense of \cite{stapleton}. Thus,
$$h^0(Z, \Omega_Z^p)=h^0(\bar{X}, \Omega_{\bar X}^p)=0$$
for every $p \geq 1$.
\end{enumerate}

\end{lemma}

\begin{remark}
\begin{enumerate}
    \item By Lemma \ref{mrc fibration invariant} and \cite[Theorem 3.7, Lecture V]{peternellbook}, the dimension of $Z$ is either $0$ or $2$.
    \item If $Z$ is a point, then $\bar{X}$ is rationally connected \cite[Corollary 5.7.1]{Kollar}, and hence $\CH_0(\bar{X})\cong\mathbb{Z}$. Corollary \ref{main corollary} then implies that every vector bundle on $X$ is trivial.
    \item If $\dim Z=2$, then Bloch's conjecture for zero-cycles on surfaces (Conjecture \ref{bloch}; see also \cite[Conjecture 11.2]{Voisin_2003} and \cite[Conjecture B]{stapleton}) implies that $\CH_0(\bar{X})\cong\mathbb{Z}$. Thus, every vector bundle on $X$ is trivial by Corollary \ref{main corollary}.
    \item Bloch's conjecture for surfaces is known when $Z$ is not of general type, that is, when $\kappa(Z)\leq1$ \cite{blochoriginal}.
    For several families of surfaces of general type with $q=p_g=0$, Bloch's conjecture is known either for the general member of the family, as for Godeaux surfaces \cite{Voisin1992}, or for specific members, such as the Barlow surface \cite{Barlow1985}; it also holds for Catanese surfaces \cite{Voisin14}. It is therefore natural to ask which surfaces of general type can occur as the base $Z$ when $\bar X$ is uniruled.
    
\end{enumerate}
    
\end{remark}

This gives the following corollary.

\begin{corollary}\label{surface bloch implies serre}
    Let $X$ be a topologically contractible smooth affine complex threefold admitting a smooth projective uniruled compactification $\bar X$ such that $D=\bar X\setminus X$ is a divisor with simple normal crossings. Then an affirmative answer to the classical Bloch conjecture for surfaces (Conjecture \ref{bloch}) implies that every algebraic vector bundle on $X$ is trivial.
\end{corollary}
\begin{remark}
    If $X$ is biholomorphic to $\mathbb{C}^3$, then $\bar{X}$ has negative Kodaira dimension by Kodaira's theorem and is therefore uniruled by three-dimensional minimal model theory; see \cite{MS1990,Kollar}. There are topologically contractible affine complex threefolds that are not biholomorphic to $\mathbb{C}^3$. It is not known whether every projective compactification of a topologically contractible affine threefold is uniruled; see Remark \ref{biholomorphic}.
\end{remark}

\subsection{Further consequences}
We collect several consequences concerning $\mathbb A^1$-connectedness, Chow groups, and vector bundles.

\begin{corollary}\label{fourfold}
    Let $X$ be a topologically contractible smooth affine complex variety of dimension $4$. Assume that $X$ is $\AA^1$-connected. If $\CH^2(X)=\CH^3(X)=0$, then every algebraic vector bundle on $X$ is trivial.
\end{corollary}
\begin{proof}
    By \cite[Corollary 2.27]{syed2024motiviccohomologycycliccoverings}, every algebraic vector bundle on a topologically contractible smooth affine complex fourfold is trivial if $\CH^i(X)=0$ for $2\leq i\leq4$ and $\HH^2_{\mathrm{Nis}}(X,\mathbf I^3)=0$. Under the additional assumption that $X$ is $\AA^1$-connected, the argument of Lemma \ref{highest chow group trivial} gives $\CH^4(X)=0$. By \cite[Proposition 3.1(c)]{brazelton2025algebraicvectorbundlesrank}, there is an epimorphism
    $$\HH^0(X,\cH_{\acute{e}t}^4(\mu_2^{\otimes4}))\twoheadrightarrow\HH^2_{\mathrm{Nis}}(X,\mathbf I^3).$$
    Finally, \cite[Lemma 4.7]{Asok+2013+39+64}, gives the required vanishing.
\end{proof}
\begin{corollary}
    If a contractible threefold $X$ is obtained by an affine modification of a rational affine threefold, then every algebraic vector bundle on $X$ is trivial. Thus all the examples in \cite{DPO} obtained by affine modifications of $\mathbb{A}^3$ have only trivial vector bundles. Some of them are stably $\mathbb{A}^1$-contractible by \cite[Theorem 3.2]{DPO}.

    Likewise, the examples in \cite[Table 1]{MS1990} that admit a smooth projective compactification with two smooth boundary components meeting along a smooth curve have only trivial algebraic vector bundles. The same holds for the examples in \cite[Table 1]{Kishimoto2005} with a smooth Fano compactification $V$ and two boundary components $D_1,D_2$ such that $K_V+D_1+D_2$ is not nef. Theorems \ref{trivial motive muller} and \ref{trivial motive Kishimoto} assert that these examples have trivial motive.
\end{corollary}
\begin{remark}
  In \cite[Theorems 5.5.2.7 and 5.5.2.8]{Asok2021}, one reduces to affine schemes using the fact that every topologically contractible smooth complex variety is $\AA^1$-weakly equivalent to a topologically contractible smooth affine scheme \cite[Lemma 5.3.1.4]{Asok2021}. Every topologically contractible smooth complex surface is affine \cite[\S2, Theorem 1]{Fujita1982OnTT}. It remains unknown whether there exists a non-affine topologically contractible smooth complex threefold \cite[Question 5.3.5.5]{Asok2021}. In dimension $4$, however, there is a non-affine $\mathbb{A}^1$-contractible smooth complex variety carrying a non-trivial vector bundle \cite[Example 5.3.2.4]{Asok2021}.
\end{remark}

\begin{remark}\phantomsection\label{KRremark}
\begin{enumerate}
   \item In \cite{KorasRussell}, Koras-Russell threefold is a smooth hypersurface in $\AA^4_\CC$ which is topologically contractible. Antieau observed that the Koras-Russell threefold of the first kind is $\AA^1$-chain connected \cite[Example 2.28]{DPO}, thus in particular it is $\AA^1$-connected \cite[Proposition 2.2.7]{AsokMorel}. Then by Corollary \ref{main corollary}, every algebraic vector bundle on Koras-Russell threefold of the first kind is trivial, which is also proved in \cite[Corollary 3.8]{Murthy02} and \cite[Corollary 3.7]{HKO} by completely different methods.

   \item Consider the complex sphere $X = \Spec \frac{\mathbb{C}[x, y, z]}{(x^2+y^2+z^2 - 1)}$. 
   It is a smooth, affine, rational surface. Indeed,
   the compactification of $X$ is isomorphic to the hypersurface $xy-zw = 0$ in $\mathbb{P}^3_{\mathbb{C}}$, which is isomorphic to $\mathbb{P}^1_{\mathbb{C}} \times_{\mathbb{C}} \mathbb{P}^1_{\mathbb{C}}$. The complex sphere $X$ is a Jouanolou device over the complex projective line $\mathbb{P}^1_{\mathbb{C}}$. Thus, in particular, the Picard group of $X$ is non-trivial. Therefore, $Y = X \times_{\mathbb{C}} \mathbb{A}^1_{\mathbb{C}}$ is a smooth rational affine threefold, and $Y$ is $\mathbb{A}^1$-weakly equivalent to $\mathbb{P}^1_{\mathbb{C}} \times_{\mathbb{C}} \mathbb{A}^1_{\mathbb{C}}$.
 
   \item In \cite{Murthy69}, Murthy constructed a smooth, rational, affine complex threefold $X = \Spec \ \frac{\mathbb{C}[x, y, t_1, t_2]}{(xy - t_1^2-t_2^3-1)}$. The threefold $X$
   is $\AA^1$-chain connected but not topologically contractible (also not $\AA^1$-contractible) and $X$ possesses a non-decomposable rank $2$ vector bundle. 
   

   \item Let $X$ be the Ramanujam surface of \cite{Ramanujam71}. Then $X$ is a topologically contractible smooth complex surface, but it is not isomorphic to the complex plane. In particular, $X$ is affine and rational \cite{GSa,GSb}. By \cite{Asok_motive}, its motive is $\mathbb Z$; hence all its positive-codimension Chow groups vanish and every algebraic vector bundle over $X$ is trivial. However, $X$ is not $\mathbb{A}^1$-connected \cite[Theorem 5.1]{ChoudhuryRoy+2024+55+80}. Consider the cylinder $Y = X \times_{\mathbb{C}}\mathbb{A}^1_{\mathbb{C}}$. Then $Y$ is a smooth complex affine threefold that is topologically contractible and has trivial Chow groups by homotopy invariance. Thus, all algebraic vector bundles over $Y$ are trivial \cite[Theorem 5.5.2.10]{Asok2021}, but $Y$ is not $\mathbb{A}^1$-connected. Another such example is a Koras--Russell threefold of the third kind; cf. Properties \ref{KRproperty}.

    \item These examples show that $\mathbb A^1$-connectedness is not necessary for the triviality of vector bundles on a rational topologically contractible affine threefold. Conversely, \cite[Example 4.14]{Asok+2013+39+64} shows that a nonzero group $\HH^0(X,\cH_{\acute{e}t}^2(\Gm))$ obstructs $\mathbb A^1$-connectedness; for proper $X$, this group is the cohomological Brauer group $H_{\acute{e}t}^2(X,\Gm)$. The proof of Theorem \ref{connectedcontractible} also shows that a topologically contractible smooth affine threefold with $\HH^0(X,\cH_{\acute{e}t}^3(\mu_l^{\otimes2}))\neq0$ for some $l$ would have $\CH^2(X)/l\neq0$ and would therefore give a counterexample to the generalized Serre question in dimension $3$.

    \end{enumerate}
\end{remark}

 In \cite[Theorem 6.6]{Asok12a}, the rank $2$ vector bundles on a smooth affine threefold $X$ over an algebraically closed field of characteristic unequal to $2$ are classified by $[X, BGL_2^{(2)}]_{\AA^1}$, where $BGL_2^{(2)}$ is the second stage of the Postnikov tower for the simplicial classifying space $BGL_2$ \cite[Section 4.1]{MV99}. We say that a pointed smooth variety $X$ is $\mathbb A^1$-$2$-connected if $\pi_0^{\mathbb A^1}(X)$ is trivial and $\pi_i^{\mathbb A^1}(X)=0$ for $i=1,2$.

\begin{theorem}\label{rational}
    If $X$ is an $\mathbb A^1$-$2$-connected rational smooth affine threefold over an algebraically closed field of characteristic not equal to $2$, then every algebraic vector bundle on $X$ is trivial.
\end{theorem}
\begin{proof}
    As $X$ is rational, a smooth projective compactification of $X$ is rational and has Chow group of $0$-cycles isomorphic to $\mathbb Z$. The localization sequence therefore gives
    $$\CH^3(X)=\CH_0(X)=0.$$
    Line bundles are represented in $\mathbb A^1$-homotopy theory by $B\Gm$, while
    $$\CH^2(X)=H^2_{\mathrm{Nis}}(X,\mathbf K_2^M)=[X,K(\mathbf K_2^M,2)]_{\mathbb A^1}.$$
    Since $B\Gm$ and $K(\mathbf K_2^M,2)$ are respectively $1$- and $2$-truncated, $\mathbb A^1$-$2$-connectedness gives
    $$\Pic(X)=[X,B\Gm]_{\mathbb A^1}=0
    \quad\text{and}\quad
    \CH^2(X)=0.$$
    The classification of vector bundles on smooth affine threefolds by their Chern classes now implies the result by \cite[Theorem 2.1(iii)]{MohanKumarMurthy} and \cite[Theorem 1]{Asok12a}.
\end{proof}
\begin{question}
    Smooth Fano threefolds are rationally connected, so if $\bar{X}$ is Fano, then $X$ has only trivial vector bundles.
    There is a classification of smooth Fano threefolds. It is  important to investigate which of these contain a contractible open threefold embedded into it. Also to note that there are non-stably rational Fano threefolds (cubic threefolds).
    Precisely, one can ask:
\begin{enumerate}
    \item If $X$ is a contractible, affine threefold over $\mathbb{C}$, then is $\bar{X}$ always Fano? 
    \item If $X$ is contractible, affine threefold over $\mathbb{C}$ such that $\bar{X}$ is Fano. Then by Theorem \ref{main theorem}, $\CH^i(X)$ is trivial for every $i \geq 1$. It is natural to ask whether $M(X)$ is trivial (or whether $\bar{X}$ is an $\mathbb{A}^1$-connected)?
\end{enumerate}

If the boundary $D=\bar{X} \setminus X$ has one component, then $\bar{X}$ is always Fano \cite{MS1990}. Based on Kishimoto's classification \cite{Kishimoto2005}, if $\bar{X} \setminus X$ has two irreducible components $D_1$ and $D_2$ such that $K_{\bar{X}}+D_1+D_2$ is not a nef divisor and $\bar{X}$ is a Fano threefold, then $M(X)$ is trivial (See Section \ref{motive computation}).
\end{question}

\section{Application to Koras-Russell threefolds}\label{SecKR}
A Koras--Russell threefold is a topologically contractible smooth affine threefold $X$ over $\mathbb C$ equipped with a hyperbolic action of $T=\mathbb G_m$, with a unique fixed point $o$, such that the affine quotient $X/\!/T$ is isomorphic to the quotient of the induced linear action on $T_oX$ \cite{KorasRussell}. The nontrivial members of this family are not isomorphic to $\mathbb A^3_\mathbb C$. We first isolate the birational argument used below.

\begin{lemma}\label{torus quotient rationality lemma}
Let $Y$ be an integral complex variety of dimension $n$ endowed with a rational action of $T=\mathbb G_m$. Assume that the generic stabilizer is finite and that
\[
    \mathbb C(Y)^T\cong \mathbb C(t_1,\ldots,t_{n-1}).
\]
Then $Y$ is rational.
\end{lemma}
\begin{proof}
Set $K=\mathbb C(Y)^T$. By Rosenlicht's theorem \cite{Rosenlicht1956}, there is a dense $T$-invariant open subset $U\subset Y$ admitting a geometric quotient $r:U\to Z$ with $\mathbb C(Z)=K$. Let $H\subset T_K$ be the generic stabilizer group scheme. By assumption $H$ is finite, so $H\cong\mu_m$ for some $m\geq1$ and $T_K/H\cong\mathbb G_{m,K}$. If $\eta=\Spec K$ is the generic point of $Z$, then, since $r$ is a geometric quotient, $T_K/H$ acts simply transitively on $U_\eta$; equivalently, $U_\eta$ is a torsor under the split torus $T_K/H$. Hilbert's Theorem 90 gives
\[
    H^1_{\mathrm{fppf}}(K,\mathbb G_m)=0,
\]
so this torsor is trivial. It follows that
\[
    \mathbb C(Y)\cong K(s)\cong\mathbb C(t_1,\ldots,t_{n-1},s),
\]
and hence $Y$ is rational.
\end{proof}

\begin{theorem}\label{all KR rational}
    Every Koras--Russell threefold is rational.
\end{theorem}
\begin{proof}
    Let $V=T_oX\cong\mathbb A^3$ be the tangent representation. After replacing $T$ by its effective quotient, its weights may be written
    \[
        w_1<0,\qquad w_2,w_3>0,
        \qquad \gcd(w_1,w_2,w_3)=1.
    \]
    Consider the weight homomorphism
    \[
        \chi:\mathbb Z^3\longrightarrow\mathbb Z,
        \qquad (m_1,m_2,m_3)\longmapsto
        w_1m_1+w_2m_2+w_3m_3.
    \]
    Choose a basis $e_1,e_2$ of $\ker(\chi)$ and $e_3\in\mathbb Z^3$ with $\chi(e_3)=1$. On the dense torus of $V$, the corresponding Laurent monomials $u,v,s$ satisfy
    \[
        \mathbb C(V)=\mathbb C(u,v,s),
        \qquad \lambda\cdot(u,v,s)=(u,v,\lambda s).
    \]
    Therefore
    \[
        \mathbb C(V)^T=\mathbb C(u,v).
    \]
    Since the weights have both signs, the quotient $V\to V/\!/T$ is geometric over a dense open subset, and hence
    \[
        \mathbb C(V/\!/T)=\mathbb C(V)^T=\mathbb C(u,v).
    \]

    By the defining quotient property of a Koras--Russell threefold,
    \[
        X/\!/T\cong V/\!/T.
    \]
    We next identify the function field of the affine quotient with the invariant rational function field. Let $\mathfrak m_o$ be the maximal ideal of the fixed point. Since $T$ is linearly reductive, an eigenbasis of $\mathfrak m_o/\mathfrak m_o^2=T_o^*X$ lifts to nonzero homogeneous functions $z_1,z_2,z_3\in\mathfrak m_o$. Their weights are the negatives of $w_1,w_2,w_3$ and hence include both signs. Since $X$ is integral, the invariant open subset
    \[
        X^\circ=\{z_1z_2z_3\ne0\},
    \]
    is nonempty and dense. Every orbit in $X^\circ$ is closed in $X$ and has finite stabilizer. Indeed, its orbit map factors through $T/H\cong\mathbb G_m$ for a finite subgroup $H\subset T$. If the orbit were not closed, the normalization of its closure would add at least one of the two points $0$ and $\infty$ to $T/H$. At $0$, the restriction of a nonzero semi-invariant of negative weight has a pole, whereas at $\infty$ the same holds for a semi-invariant of positive weight. Since all $z_i$ are regular on $X$ and nonzero along the orbit, neither endpoint can map to $X$. Thus the orbit is closed.

    Let $q:X\to Q=X/\!/T$ be the affine quotient and let $r:U\to Z$ be a Rosenlicht geometric quotient of a dense invariant open subset $U\subset X^\circ$. The invariant morphism $q|_U$ induces a dominant rational map
    \[
        \rho:Z\dashrightarrow Q.
    \]
    The fibers of an affine quotient by a reductive group contain unique closed orbits. Since the $T$-orbits in $U$ are closed, two general points of $Z$ with the same image in $Q$ represent the same orbit. Thus $\rho$ is generically injective. Over $\mathbb C$ it is therefore birational, and hence
    \[
        \mathbb C(X)^T=\mathbb C(X/\!/T)
        \cong\mathbb C(u,v).
    \]

    The generic stabilizer is finite because the action is effective. Lemma \ref{torus quotient rationality lemma} now proves that $X$ is rational.
\end{proof}

The Koras--Russell threefolds of the first kind are given by the equation
$$\{x + x^dy+z^{\alpha_2}+t^{\alpha_3 }= 0\} \subset \mathbb{A}^4_\mathbb{C} = \Spec \ \mathbb{C}[x,y,z,t],$$
where $d, \alpha_2, \alpha_3 \geq 2$ are positive integers and $\alpha_2, \alpha_3$ are relatively prime. Koras--Russell threefolds of the second kind are given by the equation
$$\{x+(x^d+z^{\alpha_2})^ly+t^{\alpha_3} = 0\} \subset \mathbb{A}^4_\mathbb{C} = \Spec \ \mathbb{C}[x,y,z,t],$$
where $l \geq 1$ and $d, \alpha_2, \alpha_3 \geq 2$ are positive integers such that $\gcd(\alpha_2, d\alpha_3) = 1$.
\begin{properties}\phantomsection\label{properties first kind}
\begin{enumerate}
    \item The first two kinds admit non-trivial $\mathbb G_a$-actions. Every non-trivial Koras--Russell threefold has a non-trivial Makar--Limanov invariant and hence is not isomorphic to $\mathbb A^3_{\mathbb C}$; it is also expected not to be biholomorphic to $\mathbb C^3$ \cite[Problem 3]{zaidenbergselected}.
    \item Murthy proved the triviality of vector bundles for the first kind \cite[Corollary 3.8]{Murthy02}. Hoyois--Krishna--{\O}stv{\ae}r proved it for the first and second kinds and showed that both are stably $\mathbb A^1$-contractible, hence have trivial motive \cite[Corollary 3.7 and Theorem 4.2]{HKO}. Moreover, a threefold of the first kind is $\mathbb A^1$-contractible \cite[Theorem 1.1]{df}, and therefore $\mathbb A^1$-chain connected \cite[Example 2.21]{DPO}; whether a threefold of the second kind is $\mathbb A^1$-contractible remains unknown.
    \item For the first two kinds, rationality can also be seen directly. For the first kind, the morphism
    $$\phi: \mathbb{A}^2_\mathbb{C} \times_\mathbb{C} \mathbb{G}_m \to X, \qquad (z,t,x) \mapsto \left(x, \frac{-z^{\alpha_2}-t^{\alpha_3}-x}{x^d}, z,t\right),$$
    is birational; for the second kind, so is
    $$\psi: \mathbb{A}^1_\mathbb{C} \times_\mathbb{C} (\mathbb{A}^2_\mathbb{C} \setminus \{x^d+z^{\alpha_2}=0\}) \to X, \qquad (t,x,z) \mapsto \left(x, \frac{-t^{\alpha_3}-x}{(x^d+z^{\alpha_2})^l}, z, t\right),$$
    and their non-trivial $\mathbb G_a$-actions imply logarithmic Kodaira dimension $-\infty$.
     \end{enumerate}
\end{properties}


There is another class of Koras--Russell threefolds, known as Koras--Russell threefolds of the third kind \cite[Example 7.6(2)]{KorasRussell}. They are given by
$$Y(d, \alpha_1, \alpha_2, \alpha_3) := \{ x + x^d y^{\alpha_1} + z^{\alpha_2} + t^{\alpha_3} = 0 \} \subset \mathbb{A}^4_\mathbb{C}=\Spec \ \mathbb{C}[x,y,z,t],$$
where $d, \alpha_1, \alpha_2,\alpha_3 \geq 2$ are positive integers such that $\alpha_1, \alpha_2, \alpha_3$ are pairwise coprime and $\gcd(d-1, \alpha_1) =1$.
\begin{properties}\phantomsection\label{KRproperty}
\begin{enumerate}
    \item A Koras--Russell threefold of the third kind is also topologically contractible \cite[Example 7.6(2)]{KorasRussell} and admits a $\mathbb{G}_m$-action. However, if $\alpha_1, \alpha_2, \alpha_3$ are sufficiently large, then $Y(d, \alpha_1, \alpha_2, \alpha_3)$ has non-negative logarithmic Kodaira dimension. Therefore, for such $\alpha_i$, the threefold $Y(d, \alpha_1, \alpha_2, \alpha_3)$ neither admits a $\mathbb{G}_a$-action nor is $\mathbb{A}^1$-connected \cite[Theorem 2.11 and Corollary 2.12]{CHOUDHURY2026}.
\item Koras--Russell threefolds of the third kind can be constructed by taking cyclic covers of Koras--Russell threefolds of the first kind \cite{syed2024motiviccohomologycycliccoverings}.
\item Syed proved that $\CH^i(Y(d, \alpha_1, \alpha_2, \alpha_3))=0$ for $i=1,2,3$ \cite[Theorem 2]{syed2026vectorbundlescertainkorasrussell}; in particular, all vector bundles on this family are trivial.

Therefore, it is natural to ask whether a Koras--Russell threefold of the third kind is stably $\mathbb{A}^1$-contractible or whether its motive is trivial in $\textbf{DM}_{gm}(\mathbb{C}, \mathbb{Z})$.
    \end{enumerate}
\end{properties}
\begin{corollary} \label{korasrussellthird}
    Every algebraic vector bundle over a Koras--Russell threefold is trivial.
\end{corollary}
\begin{proof}
    Combine Theorem \ref{all KR rational} with Corollary \ref{contractible rational vector bundles}.
\end{proof}
This answers \cite[Question 7.12]{KorasRussell} for every Koras--Russell threefold.

For the displayed third-kind family, rationality also admits the following direct birational proof.
\begin{theorem}\label{thirdkindrational}
    A Koras--Russell threefold of the third kind $Y(d, \alpha_1, \alpha_2, \alpha_3)$
    is rational.
\end{theorem}
\begin{proof}
Let
$$Y:=Y(d, \alpha_1, \alpha_2, \alpha_3) = \{ f(x,y,z,t) = x + x^d y^{\alpha_1} + z^{\alpha_2} + t^{\alpha_3} = 0 \} \subset \mathbb{A}^4_\mathbb{C}$$
be a Koras--Russell threefold of the third kind, where $d, \alpha_1, \alpha_2,\alpha_3 \geq 2$ are positive integers such that $\alpha_1, \alpha_2, \alpha_3$ are pairwise coprime and $\gcd(d-1, \alpha_1) =1$. Denote the product $\alpha_2\alpha_3$ by $M$. Since
$$\gcd(M(d-1), \alpha_1)=1
\qquad\text{and}\qquad
\gcd(\alpha_2, \alpha_3)=1,$$
there are integers $u,v,k_x,k_y$ such that
$$u \alpha_1 + v M (d-1) = 1 \text{ and } k_x \alpha_3 + k_y \alpha_2 = 1 - u.$$
Consider the rational map
$$\Theta: \mathbb{A}^4_\mathbb{C} \dashrightarrow \mathbb{A}^4_\mathbb{C}$$
defined by
\begin{equation*}
  \left\{
    \begin{aligned}
& x \mapsto & xy^{vM}z^{k_x \alpha_1M}t^{k_y\alpha_1M} \\
& y \mapsto & y^u z^{-k_xM(d-1)} t^{-k_yM(d-1)} \\
& z \mapsto & y^{v\alpha_3} z^{k_x \alpha_1\alpha_3 +1} t^{k_y\alpha_1\alpha_3} \\ 
& t \mapsto & y^{v \alpha_2} z^{k_x\alpha_1\alpha_2} t^{k_y\alpha_1\alpha_2+1}
    \end{aligned}
    \right.
    \end{equation*}
    The $4 \times 4$ matrix $W$
$$W = 
\begin{pmatrix}
1 & vM & k_x \alpha_1 M & k_y \alpha_1 M \\
0 & u & -k_x M (d-1) & -k_y M (d-1) \\
0 & v\alpha_3 & k_x \alpha_1 \alpha_3 + 1 & k_y \alpha_1 \alpha_3 \\
0 & v\alpha_2 & k_x \alpha_1 \alpha_2 & k_y \alpha_1 \alpha_2 + 1
\end{pmatrix}$$
    has determinant $1$. Indeed, expanding along its first column and using $M=\alpha_2\alpha_3$ gives
    \begin{align*}
        \det(W)
        & =u\bigl(1+\alpha_1(k_x\alpha_3+k_y\alpha_2)\bigr)
        +vM(d-1)(k_x\alpha_3+k_y\alpha_2)\\
        & =u+(1-u)\bigl(u\alpha_1+vM(d-1)\bigr)=1.
    \end{align*}
    Thus $W \in SL_4(\mathbb{Z})$, and the restriction of $\Theta$ to the dense torus $(\mathbb G_m)^4$ is an algebraic torus automorphism. Therefore, $\Theta$ is birational. 
We claim that $\Theta$ restricts to a rational map $\phi: X \dashrightarrow Y$, where $X$ is the Koras--Russell threefold of the first kind given by the equation
$$X \equiv \{g(x,y,z,t) =x + x^dy+z^{\alpha_2}+t^{\alpha_3 }= 0\} \subset \mathbb{A}^4_\mathbb{C}.$$
Indeed,
    \begin{align*}
 & f(xy^{vM}z^{k_x \alpha_1M}t^{k_y\alpha_1M}, y^u z^{-k_xM(d-1)} t^{-k_yM(d-1)},y^{v\alpha_3} z^{k_x \alpha_1\alpha_3 +1} t^{k_y\alpha_1\alpha_3}, y^{v \alpha_2} z^{k_x\alpha_1\alpha_2} t^{k_y\alpha_1\alpha_2+1})\\
 = & xy^{vM}z^{k_x \alpha_1M}t^{k_y\alpha_1M}+ (xy^{vM}z^{k_x \alpha_1M}t^{k_y\alpha_1M})^d (y^u z^{-k_xM(d-1)} t^{-k_yM(d-1)})^{\alpha_1}+ \\ 
 &(y^{v\alpha_3} z^{k_x \alpha_1\alpha_3 +1} t^{k_y\alpha_1\alpha_3})^{\alpha_2} + (y^{v \alpha_2} z^{k_x\alpha_1\alpha_2} t^{k_y\alpha_1\alpha_2+1})^{\alpha_3}\\
= & y^{vM}z^{k_x\alpha_1M}t^{k_y\alpha_1M} g(x,y,z,t)
    \end{align*}
    Thus $\Theta$ restricts to a rational map
    $$\phi: X \dashrightarrow Y.$$
    Let $T=(\mathbb G_m)^4$. The intersections $X\cap T$ and $Y\cap T$ are dense open subsets of $X$ and $Y$, respectively. Since $W\in SL_4(\mathbb Z)$, the restriction $\Theta|_T:T\to T$ is an algebraic torus automorphism. The identity above shows that it maps $X\cap T$ into $Y\cap T$. Conversely, if $q\in Y\cap T$ and $p=\Theta^{-1}(q)$, then
    $$0=f(q)=f(\Theta(p))=y(p)^{vM}z(p)^{k_x\alpha_1M}t(p)^{k_y\alpha_1M}g(p).$$
    The monomial factor is invertible on $T$, so $g(p)=0$. Hence $\Theta^{-1}$ maps $Y\cap T$ into $X\cap T$, and $\Theta$ restricts to an isomorphism between the dense open subsets $X\cap T$ and $Y\cap T$. In particular, $\phi$ is dominant and, because $W \in SL_4(\mathbb{Z})$, the induced map
    $$\phi^*: \mathbb{C}(Y) \to \mathbb{C}(X)$$
    between the function fields of $X$ and $Y$ is an isomorphism. Therefore, $\phi$ is birational. Since a Koras--Russell threefold of the first kind is rational, $Y$ is rational.
    
\end{proof}
For sufficiently large exponents, Theorem \ref{thirdkindrational} and Properties \ref{KRproperty}(i) provide rational topologically contractible threefolds without a non-trivial $\mathbb G_a$-action; compare \cite[Corollary 2.6]{kaliman2002bfcactionscontractiblethreefolds}.

\section{Motive computations for selected compactifications} \label{motive computation}
Let $\mathbf{DM}_{gm}(\mathbb C,\mathbb Z)$ denote Voevodsky's triangulated category of geometric motives with integral coefficients.
In dimension $1$, every topologically contractible complex curve is $\mathbb A^1$. In dimension $2$, contractibility implies rationality \cite{GSa,GSb} and trivial Voevodsky motive \cite{Asok_motive}, but not isomorphism with $\mathbb A^2$, as shown by the Ramanujam surface \cite{Ramanujam71}. On the other hand, a smooth surface homeomorphic to $\mathbb C^2$ is isomorphic to $\mathbb A^2_\mathbb C$ \cite{Ramanujam71}; hence there are no exotic affine planes.

The situation is completely different in dimension $3$ and higher. By the Dimca--Ramanujam theorem \cite[Theorem 3.2]{zaidenberglecture}, a topologically contractible smooth affine complex variety of dimension $n\geq 3$ is diffeomorphic to $\mathbb{C}^n$.
A complex affine threefold $X$ is called an exotic $\mathbb A^3$ if it is diffeomorphic to $\mathbb C^3$ but not algebraically isomorphic to $\mathbb A^3_\mathbb C$ \cite[Definition 2.3]{Kishimoto2005}. Nontrivial Koras--Russell threefolds form an important family of such varieties \cite{KorasRussell}. Those of the first and second kinds have trivial motive \cite{HKO}. The motive of a Koras--Russell threefold of the third kind is not known to be trivial, although Corollary \ref{korasrussellthird} shows that all its algebraic vector bundles are trivial; for certain third-kind threefolds this was also proved directly in \cite{syed2026vectorbundlescertainkorasrussell}.


    \begin{question} \label{general motive question}
     Let $X$ be a topologically contractible smooth affine complex threefold. Is $M(X)\cong\mathbb Z$ in $\mathbf{DM}_{gm}(\mathbb C,\mathbb Z)$?
    \end{question}
Fix a smooth projective compactification $\bar X$ whose boundary $D=\bar X\setminus X$ is a simple normal crossings divisor, and write $D_1,\ldots,D_r$ for its irreducible components. M\"uller--Stach classified the pairs $(\bar X,D)$ for which $X$ is biholomorphic to $\mathbb C^3$, $r=2$, and $D_1\cap D_2$ is a smooth curve \cite{MS1990}. Biholomorphism with $\mathbb C^3$ is strictly stronger than topological contractibility; see Remark \ref{biholomorphic}.
Kishimoto classified the corresponding pairs when $X$ is topologically contractible, $r=2$, $\bar X$ is Fano, and $K_{\bar X}+D_1+D_2$ is not nef \cite{Kishimoto2005}.
The computations below give an affirmative answer to Question \ref{general motive question} for the pairs in M\"uller--Stach's classification and for the pairs in Kishimoto's classification other than possibly Type XIII. The same conclusion holds in Type XIII whenever its affine-modification divisor is smooth, by Theorem \ref{trivial motive Kishimoto}.
When $D$ has one irreducible component (not necessarily smooth), the Picard group of $\bar{X}$ is $\ZZ$. If, in addition, $\bar X$ is Fano, then \cite[Corollary 2.1]{Kishimoto2005} shows that a contractible affine threefold $X$ embedded in $\bar X$ is isomorphic to $\AA^3$. We can summarize this case as follows:

\begin{theorem} 
Let $X$ be a topologically contractible smooth affine complex threefold with smooth Fano compactification $\bar X$. If $D=\bar X\setminus X$ is irreducible, then $X\cong\mathbb A^3_{\mathbb C}$. In particular, $M(X)\cong\mathbb Z$.
\end{theorem}

Moreover, for compactifications of affine space with smooth boundary, this result extends to dimensions up to $6$:

\begin{theorem}
Suppose that $X$ is biholomorphic to $\mathbb C^n$, where $n\leq6$, and admits a smooth projective compactification $\bar X$ with smooth irreducible boundary $D$. Then $(\bar X,D)\cong(\mathbb P^n,\mathbb P^{n-1})$, so $X\cong\mathbb A^n_{\mathbb C}$ and $M(X)\cong\mathbb Z$.
\end{theorem}
\begin{proof}
    By \cite{vandeVen1962,Fujita1980OnTC}, since the boundary $D$ is smooth, the only compactification of $X$ is $\PP^n$ for $n\leq 6$, and $\partial\bar{X}=\PP^{n-1}$.
\end{proof}

We first record a splitting observation that will be used repeatedly. Call an object of $\textbf{DM}_{gm}(\mathbb C,\mathbb Z)$ \emph{pure Tate} if it is a finite direct sum of motives $\mathbb Z(i)[2i]$.

\begin{lemma}\label{Tate splitting lemma}
Let $f:P\to Q$ be a morphism between pure Tate motives. Suppose that the Betti realization of $f$ is injective in every degree and has torsion-free cokernel. Then $f$ is split injective in $\textbf{DM}_{gm}(\mathbb C,\mathbb Z)$, and its cone is the pure Tate motive determined by the graded cokernel of its Betti realization.
\end{lemma}
\begin{proof}
Write
$$P=\bigoplus_i\mathbb Z(i)[2i]^{\oplus a_i}
\qquad\text{and}\qquad
Q=\bigoplus_i\mathbb Z(i)[2i]^{\oplus b_i}.$$
For the motive of a point,
$$\operatorname{Hom}\bigl(\mathbb Z(i)[2i],\mathbb Z(j)[2j]\bigr)
\cong H^{2(j-i),j-i}(\Spec\mathbb C,\mathbb Z)$$
is $\mathbb Z$ if $i=j$ and is zero otherwise. Thus $f$ is given, for each $i$, by an integral matrix
$$f_i:\mathbb Z^{a_i}\longrightarrow\mathbb Z^{b_i},$$
and Betti realization gives the same matrices. By hypothesis, each $f_i$ is injective with torsion-free cokernel. Its Smith normal form therefore has only diagonal entries equal to $1$, so $f_i$ is equivalent over $\mathbb Z$ to the standard inclusion of a direct summand. The same changes of basis define automorphisms of $P$ and $Q$ in $\textbf{DM}_{gm}(\mathbb C,\mathbb Z)$. Hence $f$ is split injective and its cone is the direct sum of the complementary Tate summands.
\end{proof}

In the computations below, $\operatorname{coker}(f)$ denotes the complementary summand of the target of a split monomorphism $f$; equivalently, it is the cone of $f$ under the identification furnished by Lemma \ref{Tate splitting lemma}.

\begin{theorem} \label{trivial motive muller}
   Suppose that $X$ is biholomorphic to $\mathbb{C}^3$, and that $D$ consists of two smooth projective surfaces (so $r = 2$) intersecting along a smooth curve, as in \cite[Main Theorem]{MS1990}. Then $M(X)$ is isomorphic to $\mathbb{Z}$ in $\textbf{DM}_{gm}(\mathbb{C}, \mathbb{Z})$; hence, every vector bundle over $X$ is trivial.
\end{theorem}
\begin{proof} 
 We follow the classification of pairs $(\bar{X}, D)$ in which $D$ consists of two smooth irreducible components, as given in \cite[Table 1 and Main Theorem]{MS1990}.

First consider any of the four rational Fano threefolds
$$\mathbb P^3,\quad\mathbb Q^3,\quad\mathbb V_5,\quad\mathbb A_{22}.$$
Each has $b_3=0$, Picard group $\mathbb Z$, and pure Tate motive
$$M(Y) \cong M(\mathbb{P}^3) \cong \mathbb{Z} \oplus \mathbb{Z}(1)[2] \oplus \mathbb{Z}(2)[4] \oplus \mathbb{Z}(3)[6].$$
 
Indeed, the bounded derived category $\mathcal{D}^b(Y)$ admits a full exceptional collection of the expected length \cite[Corollary 2.1]{MarcolliTabuada+2015+153+167}; see also \cite[Proposition 1]{ciolli}. Therefore, the Chow motive of $Y$ is of Lefschetz type \cite[Theorem 2.2]{Gorchin}. We compute the motive in the category $DM_{gm}(\mathbb{C}, \mathbb{Z})$ \cite{Voevodsky+2011+188+238}.

The localization sequence for compactly supported motives \cite[Proposition 4.1.5]{Voevodsky+2011+188+238} gives a triangle of the form
\begin{equation*}
    M^c(\partial\bar{X})\to M^c(\bar{X})\to M^c(X)\to M^c(\partial\bar{X})[1].
\end{equation*}
Writing $\partial\bar X=A_1\cup A_2$, the Mayer--Vietoris sequence \cite[Proposition 4.1.1]{Voevodsky+2011+188+238} gives
\begin{equation*}
    M(A_1\cap A_2)\to M(A_1)\oplus M(A_2)\to M(\partial\bar X)\to M(A_1\cap A_2)[1].
\end{equation*}

Since $\bar{X}$ and $D$ are projective, it follows
that $M^c(\bar{X})\cong M(\bar{X})$ and $M^c(\partial\bar{X})\cong M(\partial\bar{X})$. We then describe the motive of the compactification $M(\bar{X})$ and the boundary $M(\partial\bar{X})$. 

\textbf{Case 1:} Following \cite[Table 1]{MS1990}, $\partial\bar{X} = A_1 \cup A_2$, where $A_1$ is isomorphic to the projectivization of a rank $2$ vector bundle over $\mathbb{P}^1$ and $A_2$ is isomorphic to $\mathbb{P}^2$. The intersection $A_1\cap A_2$ is isomorphic to $\mathbb{P}^1$. By the projective bundle formula for geometric motives \cite[Corollary 4.1.11]{Voevodsky+2011+188+238},
\begin{align*}
     M(\bar{X}) &\cong \bigoplus_{i=0}^2 M(\mathbb{P}^1)(i)[2i]  \\
     & \cong \bigoplus_{i=0}^2 \mathbb{Z}(i)[2i] \oplus \bigoplus_{i=0}^2 (\mathbb{Z}(1)[2])(i)[2i]
     \\
     &
     \cong \mathbb{Z} \oplus \mathbb{Z}(1)[2]^{\oplus 2} \oplus \mathbb{Z}(2)[4]^{\oplus 2} \oplus \mathbb{Z}(3)[6].
\end{align*}
$M(A_1 \cap A_2) \cong M(\mathbb{P}^1) \cong \mathbb{Z} \oplus \mathbb{Z}(1)[2]$, $M(A_1) \cong M(\mathbb{P}^1) \oplus M(\mathbb{P}^1)(1)[2] \cong \mathbb{Z} \oplus \mathbb{Z}(1)[2]^{\oplus 2} \oplus \mathbb{Z}(2)[4]$ and $M(A_2) \cong M(\mathbb{P}^2) \cong \mathbb{Z} \oplus \mathbb{Z}(1)[2]\oplus \mathbb{Z}(2)[4]$.

Thus, the Mayer--Vietoris triangle has the form
\begin{equation*}
    \mathbb{Z} \oplus \mathbb{Z}(1)[2] \to \mathbb{Z}^{\oplus 2} \oplus \mathbb{Z}(1)[2]^{\oplus 3} \oplus \mathbb{Z}(2)[4]^{\oplus 2} \to M(\partial\bar{X}) \to (\mathbb{Z} \oplus \mathbb{Z}(1)[2])[1]
\end{equation*}
    All motives in the first two terms are pure Tate. Under Betti realization, the first morphism becomes the first map in the topological Mayer--Vietoris sequence. Since $H^*(\partial\bar X,\mathbb Z)\cong H^*(\bar X,\mathbb Z)$ in degrees at most $4$ and the cohomology of $\bar X$ is free and concentrated in even degrees in this case, this map is injective with torsion-free cokernel. Lemma \ref{Tate splitting lemma} therefore shows that the first morphism is split injective, and
    $$M(\partial \bar{X}) \cong \operatorname{coker}(\mathbb{Z} \oplus \mathbb{Z}(1)[2] \to \mathbb{Z}^{\oplus 2} \oplus \mathbb{Z}(1)[2]^{\oplus 3} \oplus \mathbb{Z}(2)[4]^{\oplus 2}) \cong \mathbb{Z} \oplus \mathbb{Z}(1)[2]^{\oplus 2} \oplus \mathbb{Z}(2)[4]^{\oplus 2}.$$
Put
$$P:=\mathbb{Z} \oplus \mathbb{Z}(1)[2]^{\oplus 2} \oplus \mathbb{Z}(2)[4]^{\oplus 2}.$$
Thus, the localization sequence takes the form
\begin{equation*}
    P\longrightarrow P\oplus\mathbb Z(3)[6]\longrightarrow M^c(X)\longrightarrow P[1].
\end{equation*}
  The first map is a morphism between pure Tate motives. Its Betti realization is an isomorphism in degrees below $6$ and has cokernel $\mathbb Z$ in degree $6$: this follows from the topological localization sequence and the fact that $X$ is contractible. Lemma \ref{Tate splitting lemma} therefore shows that it is split injective, and
  \begin{align*}
      M^c(X) & \cong \operatorname{coker}\bigl(P\longrightarrow P\oplus\mathbb Z(3)[6]\bigr) \\
      & \cong \mathbb{Z}(3)[6].
  \end{align*}

\textbf{Case 2a:}
In this case $\bar{X}$ is either a projective bundle over $\mathbb{P}^2$ or $\bar{X}$ is a smooth divisor of type $(2,1)$ in $\mathbb{P}^2 \times \mathbb{P}^2$. If $\bar{X}$ is given by a bihomogeneous polynomial of $x_0,x_1,x_2;y_0,y_1,y_2$ of bidegree $(2,1)$, then also $\bar{X}$ is a $\mathbb{P}^1$-bundle over $\mathbb{P}^2$. Indeed, write the equation as $q_0(x)y_0+q_1(x)y_1+q_2(x)y_2=0$. The quadrics $q_0,q_1,q_2$ have no common zero: at a common zero, a nonzero linear relation among their differentials would produce a singular point of $\bar X$. Hence every fiber of the first projection is a hyperplane in $\mathbb P^2$, and the projection $\bar{X} \to \mathbb{P}^2$ given by $([x_0:x_1:x_2], [y_0:y_1:y_2]) \mapsto [x_0:x_1:x_2]$ is a $\mathbb{P}^1$-bundle.

The projective bundle formula yields
\begin{align*}
    M(\bar{X}) &\cong M(\mathbb{P}^2) \oplus M(\mathbb{P}^2)(1)[2] \\
            & \cong \mathbb{Z} \oplus \mathbb{Z}(1)[2]^{\oplus 2} \oplus \mathbb{Z}(2)[4]^{\oplus 2} \oplus \mathbb{Z}(3)[6]
            \end{align*}
In this case the boundary is $\partial\bar X=A_1\cup A_2$, where $A_1$ and $A_2$ are smooth projective rational surfaces and $A_1\cap A_2\cong\mathbb P^1$. Hence
$$M(A_1) \cong \mathbb{Z} \oplus \mathbb{Z}(1)[2]^{\oplus r} \oplus \mathbb{Z}(2)[4], \text{ and }$$
$$M(A_2) \cong \mathbb{Z} \oplus \mathbb{Z}(1)[2]^{\oplus s} \oplus \mathbb{Z}(2)[4], $$
for some integers $r, s \geq 1$.
The Mayer--Vietoris triangle becomes
$$\mathbb{Z} \oplus \mathbb{Z}(1)[2] \to \mathbb{Z}^{\oplus 2} \oplus \mathbb{Z}(1)[2]^{\oplus (r+s)} \oplus \mathbb{Z}(2)[4]^{\oplus 2} \to M(\partial \bar{X}) \to (\mathbb{Z} \oplus \mathbb{Z}(1)[2])[1].$$
The first map is a morphism between pure Tate motives. As in Case 1, the topological Mayer--Vietoris sequence identifies its Betti realization with an injection having torsion-free cokernel. Lemma \ref{Tate splitting lemma} therefore shows that the first morphism is split injective, and
  \begin{align*}
      M(\partial \bar{X}) & \cong \operatorname{coker}(\mathbb{Z} \oplus \mathbb{Z}(1)[2] \to \mathbb{Z}^{\oplus 2} \oplus \mathbb{Z}(1)[2]^{\oplus (r+s)} \oplus \mathbb{Z}(2)[4]^{\oplus 2}) \\
      & \cong \mathbb{Z} \oplus \mathbb{Z}(1)[2]^{\oplus (r+s-1)} \oplus \mathbb{Z}(2)[4]^{\oplus 2}.
  \end{align*}
  Thus, the rank of $\Pic(\partial \bar{X})$ is $r+s-1$. The morphism
$$H^2(\bar{X}, \mathbb{Z}) \to H^2(\partial \bar{X})$$
is an isomorphism \cite[Section 3]{MS1990}. Since $\Pic(\bar{X}) \cong \mathbb{Z} \oplus \mathbb{Z}$, we have
$$r+s = 3.$$
Thus $M(\partial\bar X)\cong P$ and $M(\bar X)\cong P\oplus\mathbb Z(3)[6]$. The localization and Tate-splitting argument of Case 1 applies verbatim and gives
$$M^c(X)\cong\mathbb Z(3)[6].$$
    


\textbf{Case 2b:} In this case $\bar{X}$ is also a projective bundle over $\mathbb{P}^2$, the surfaces $A_1$ and $A_2$ are rational (in particular, $A_1 \cong \mathbb{P}^2$ and $A_2 \cong \mathbb{F}_a$), and $A_1 \cap A_2 \cong \mathbb{P}^1$. Hence, this case is included in Case 2a, and we conclude that
$$M^c(X) \cong \mathbb{Z}(3)[6].$$
\textbf{Case 3:} Here $\bar X$ is the blow-up $\pi:\bar X\to Y$ along a point or a smooth curve, where $Y$ is one of $\mathbb P^3$, $\mathbb Q^3$, $\mathbb V_5$, or $\mathbb A_{22}$. Write $\partial\bar X=A_1\cup A_2$.

If $\pi$ is the blow-up of $Y$ at a $k$-point, then $A_1$ and $A_2$ are rational and $A_1\cap A_2\cong\mathbb P^1$. The blow-up formula \cite[Proposition 3.5.3]{Voevodsky+2011+188+238} gives
$M(\bar X)\cong P\oplus\mathbb Z(3)[6]$ and $\Pic(\bar X)\cong\mathbb Z^2$. The boundary calculation of Case 2a gives $M(\partial\bar X)\cong P$, so the localization argument of Case 1 yields
$$M^c(X)\cong\mathbb Z(3)[6].$$

\textbf{Case 3a:} In this case, $\pi$ is the blow-up along a smooth curve $S \subset Y$, where $S \cong \mathbb{P}^1$; moreover, $A_1$ and $A_2$ are rational surfaces and $A_1 \cap A_2 \cong \mathbb{P}^1$.

The same blow-up formula gives
$M(\bar X)\cong P\oplus\mathbb Z(3)[6]$. Proceeding as above gives
$$M^c(X) \cong \mathbb{Z}(3)[6].$$
\textbf{Case 3b:} Here the center of $\pi$ is a smooth curve $S$, possibly of positive genus. After renumbering, the exceptional divisor is the boundary component $A_2$, a $\mathbb P^1$-bundle over $S$. The other component $A_1$ is a smooth rational surface, and $A_1\cap A_2\cong S$.

The projective bundle formula gives
$$M(E) \cong M(A_2)  \cong M(S) \oplus M(S)(1)[2].$$
and therefore $\operatorname{rk}\operatorname{NS}(A_2)=2$.

By \cite[Proposition 3.5.3]{Voevodsky+2011+188+238}, we have
$$M(\bar{X}) \cong \mathbb{Z} \oplus \mathbb{Z}(1)[2] \oplus \mathbb{Z}(2)[4] \oplus \mathbb{Z}(3)[6] \oplus M(S)(1)[2].$$
$$M(S) \to \mathbb{Z} \oplus \mathbb{Z}(1)[2] \oplus \mathbb{Z}(2)[4] \oplus M(S) \oplus M(S)(1)[2] \to M(\partial\bar{X}) \to M(S)[1]$$
Thus, $\operatorname{rk}\Pic(\bar{X}) = 2$.

Since $A_1$ is rational,
$$M(A_1) \cong \mathbb{Z} \oplus \mathbb{Z}(1)[2]^{\oplus r} \oplus \mathbb{Z}(2)[4].$$
The topological Mayer--Vietoris sequence gives
$$\operatorname{rank}H^2(\partial\bar X,\mathbb Z)
=\operatorname{rank}H^2(A_1,\mathbb Z)+\operatorname{rank}H^2(A_2,\mathbb Z)-\operatorname{rank}H^2(S,\mathbb Z)
=r+1.$$
Since the morphism
$$H^2(\bar{X}, \mathbb{Z}) \to H^2(\partial \bar{X},\mathbb Z)$$
is an isomorphism \cite[Section 3]{MS1990}, we have $2=r+1$.
Thus $r=1$ and
$$M(A_1)\cong\mathbb Z\oplus\mathbb Z(1)[2]\oplus\mathbb Z(2)[4].$$
Here the Mayer--Vietoris triangle has the form
\begin{equation*}
    M(S)\to \mathbb{Z} \oplus \mathbb{Z}(1)[2] \oplus \mathbb{Z}(2)[4] \oplus M(S) \oplus M(S)(1)[2] \to M(\partial\bar{X})\to M(S)[1].
    \end{equation*}
The inclusion $S=A_1\cap A_2\hookrightarrow A_2$ is a section of the projective bundle $A_2\to S$. Consequently, the component $M(S)\to M(A_2)$ of the first morphism has a retraction. After an automorphism of $M(A_1)\oplus M(A_2)$, the Mayer--Vietoris morphism is therefore the inclusion of the $M(S)$-summand in
$$M(A_2)\cong M(S)\oplus M(S)(1)[2].$$
Thus it is split injective, and
\begin{align*}
    M(\partial \bar{X}) & \cong \operatorname{coker}(M(S)\to \mathbb{Z} \oplus \mathbb{Z}(1)[2] \oplus \mathbb{Z}(2)[4] \oplus M(S) \oplus M(S)(1)[2]) \\
   & \cong \mathbb{Z} \oplus \mathbb{Z}(1)[2] \oplus \mathbb{Z}(2)[4] \oplus M(S)(1)[2]
\end{align*}

Put
$$Q:=\mathbb{Z} \oplus \mathbb{Z}(1)[2] \oplus \mathbb{Z}(2)[4] \oplus M(S)(1)[2].$$
Thus, the localization sequence takes the form
\begin{align*}
  Q\longrightarrow Q\oplus\mathbb Z(3)[6]\longrightarrow M^c(X)\longrightarrow Q[1].
\end{align*}
Under the blow-up decomposition
$$M(\bar X)\cong M(Y)\oplus M(S)(1)[2],$$
the summand $M(S)(1)[2]$ is induced by the exceptional divisor $A_2$. Hence the boundary morphism restricts to an isomorphism on the common $M(S)(1)[2]$-summand. Cancelling this summand reduces the first map to a morphism between the pure Tate motives
$$\mathbb Z\oplus\mathbb Z(1)[2]\oplus\mathbb Z(2)[4]
\longrightarrow
\mathbb Z\oplus\mathbb Z(1)[2]\oplus\mathbb Z(2)[4]\oplus\mathbb Z(3)[6].$$
Its Betti realization is an isomorphism below degree $6$ with cokernel $\mathbb Z$ in degree $6$. Lemma \ref{Tate splitting lemma} shows that the reduced morphism, and hence the original boundary morphism, is split injective. Therefore,
\begin{align*}
    M^c(X) & \cong \operatorname{coker}\bigl(Q\longrightarrow Q\oplus\mathbb Z(3)[6]\bigr) \\
    & \cong \mathbb{Z}(3)[6].
\end{align*}


Since $X$ is smooth of dimension $3$, motivic duality applied to $M^c(X)\cong\mathbb Z(3)[6]$ gives $M(X)\cong\mathbb Z$ \cite[Theorem 4.3.7.3]{Voevodsky+2011+188+238}. Hence $\CH^i(X)=0$ for every $i>0$, and every algebraic vector bundle on $X$ is trivial by Theorem \ref{threefold}.
\end{proof}
\begin{remark}\label{affinemodification}
    In \cite{Kishimoto2005}, Kishimoto presented a classification of compactifications of topologically contractible threefolds with two boundary components. Notably, the classification contains exotic $\AA^3$'s. Following Kishimoto's classification \cite[Remark 1.3]{Kishimoto2005}, in case (1), if $X$ is isomorphic to the affine modification of $\mathbb{A}^3$ along the surface $S_{a,b}$ (a tom Dieck--Petrie surface) at the center of a smooth point of $S_{a,b}$, then $X$ is stably $\mathbb{A}^1$-contractible \cite[Theorem 3.2]{DPO}, because $M(S_{a,b})$ is trivial: indeed, $S_{a,b}$ is topologically contractible by \cite[Proposition 4]{Asok_motive}. \par
    Also, in case (2), the threefold $X$ is isomorphic to the affine modification of $\mathbb{A}^3$ along $S$ (where $S$ is a cylinder over the cuspidal curve) at the center of a smooth point; this threefold is in turn isomorphic to
    $$X \cong S_{a,b} \times \mathbb{A}^1,$$
    so in this case also, $X$ is stably $\mathbb{A}^1$-contractible by \cite[Example 2.25]{DPO}, but $X$ is not $\mathbb{A}^1$-connected. 

    The following theorem combines these descriptions with a Gysin calculation for Type XIII. It proves that the contractible threefolds $X$ appearing in Types XII and XXII of \cite[Table 1]{Kishimoto2005} have trivial motive, and gives the same conclusion in Type XIII whenever the divisor in the affine-modification description is smooth.


\end{remark}

\begin{proposition}\label{homology Gm surface}
    Let $S$ be a smooth affine surface over $\CC$ such that
    \[
        H_0(S,\ZZ)\cong H_1(S,\ZZ)\cong\ZZ,
        \qquad H_i(S,\ZZ)=0\quad(i>1).
    \]
    Then $S$ is rational and
    \[
        M(S)\cong M(\mathbb G_m)
        \cong\ZZ\oplus\ZZ(1)[1]
    \]
    in $\textbf{DM}_{gm}(\CC,\ZZ)$.
\end{proposition}
\begin{proof}
    Choose a smooth projective compactification $S\subset Y$ whose boundary $B=Y\setminus S$ is a simple normal crossing divisor. Poincar\'e duality gives
    \[
        H_c^i(S,\ZZ)\cong
        \begin{cases}
            \ZZ,&i=3,4,\\
            0,&i\ne3,4.
        \end{cases}
    \]
    The localization sequence for $S\subset Y$ in degrees zero and one shows that $B$ is connected.

    We first prove that $S$ is rational. The mixed Hodge structure on $H^1(S,\QQ)$ has dimension one. Its subspace $W_1H^1(S,\QQ)$ is a polarizable Hodge structure of weight one and therefore has even dimension. Consequently,
    \[
        W_1H^1(S,\QQ)=0.
    \]
    On the other hand, the natural map
    \[
        H^1(Y,\QQ)\longrightarrow H^1(S,\QQ)
    \]
    is injective and its image is $W_1H^1(S,\QQ)$ \cite{PMIHES_1971__40__5_0}. Hence $q(Y)=0$. The surface $S$ is affine, its Euler characteristic is zero, and its boundary $B$ is connected. The affine case of \cite[\S5]{GurjarParameswaran1995} therefore shows that the minimal model of $Y$ is either $\PP^2$ or a ruled surface. In the latter case, the genus of the base of the ruling equals $q(Y)$ and is therefore zero. Thus $Y$, and hence $S$, is rational.

    We now compute the motive. Since $Y$ is a smooth projective rational surface, there is an integral decomposition
    \[
        M(Y)\cong\ZZ\oplus\ZZ(1)[2]^{\oplus\rho}
        \oplus\ZZ(2)[4],
        \qquad \rho=b_2(Y).
    \]
    The localization sequence and the above computation of $H_c^*(S,\ZZ)$ also give $H^1(B,\ZZ)=0$. It follows from the normalization sequence that every irreducible component of $B$ is a rational curve and that the dual graph of $B$ is a tree. If $r$ is the number of irreducible components of $B$, then
    \[
        M(B)\cong\ZZ\oplus\ZZ(1)[2]^{\oplus r}.
    \]
    Since $Y$ is rational, $H_3(Y,\mathbb Z)=0$; moreover, the preceding paragraph gives $H_1(B,\mathbb Z)=0$. The Borel--Moore localization sequence therefore contains
    \[
        0\longrightarrow H_3^{\mathrm{BM}}(S,\ZZ)
        \longrightarrow H_2(B,\ZZ)\longrightarrow H_2(Y,\ZZ)
        \longrightarrow H_2^{\mathrm{BM}}(S,\ZZ)\longrightarrow0.
    \]
    Here $H_3^{\mathrm{BM}}(S,\ZZ)\cong H^1(S,\ZZ)\cong\ZZ$ and
    $H_2^{\mathrm{BM}}(S,\ZZ)\cong H^2(S,\ZZ)=0$. Hence
    \[
        0\longrightarrow\ZZ\longrightarrow\ZZ^r
        \longrightarrow\ZZ^\rho\longrightarrow0,
    \]
    so $r=\rho+1$ and the map $H_2(B,\ZZ)\to H_2(Y,\ZZ)$ is a split surjection. Under the displayed Tate decompositions, the morphism
    $M(B)\to M(Y)$ is the identity on the $\ZZ$-summand, is represented by this split-surjective integral matrix on the $\ZZ(1)[2]$-summands, and has no other components for degree reasons. The localization triangle
    \[
        M(B)\longrightarrow M(Y)\longrightarrow M^c(S)
        \longrightarrow M(B)[1]
    \]
    therefore gives
    \[
        M^c(S)\cong\ZZ(1)[3]\oplus\ZZ(2)[4].
    \]
    Finally, motivic duality for the smooth surface $S$ yields
    \[
        M(S)\cong M^c(S)^\vee(2)[4]
        \cong\ZZ\oplus\ZZ(1)[1],
    \]
    as claimed.
\end{proof}

\begin{theorem} \label{trivial motive Kishimoto}
    Let $X=\bar X\setminus(D_1\cup D_2)$ be a topologically contractible smooth affine threefold appearing in \cite[Table 1]{Kishimoto2005}, where $\bar X$ is a smooth Fano threefold and $K_{\bar X}+D_1+D_2$ is not nef. Then $M(X)$ is isomorphic to $\mathbb{Z}$ in $\textbf{DM}_{gm}(\mathbb{C}, \mathbb{Z})$ in every case except possibly Type XIII. The same conclusion holds in Type XIII whenever the divisor in its affine-modification description is smooth.
    In each of these cases, every vector bundle over $X$ is trivial.
\end{theorem}
\begin{proof}
    In \cite[Table 1]{Kishimoto2005}, there is a complete classification of contractible complex affine threefolds satisfying the above condition. In most cases, $X\cong \AA^3$. As indicated in \cite[Remark 1.3]{Kishimoto2005}, threefolds of Types XII, XIII, and XXII in the classification may be exotic $\AA^3$'s.

    By \cite[Lemma 5.4 and Remark 1.3]{Kishimoto2005}, Types XII and XXII are the two affine-modification cases recalled in Remark \ref{affinemodification}. In the first case, \cite[Theorem 3.2]{DPO} shows that $X$ is stably $\mathbb A^1$-contractible. In the second case, Kishimoto's description gives
    $$X\cong S_{a,b}\times\mathbb A^1,$$
    where the tom Dieck--Petrie surface $S_{a,b}$ is a smooth topologically contractible affine surface. Hence $M(S_{a,b})\cong\mathbb Z$ by \cite{Asok_motive}, and homotopy invariance gives $M(X)\cong\mathbb Z$. Stable $\mathbb A^1$-contractibility also implies that the motive is $\mathbb Z$ in the first case. Thus $M(X)\cong\mathbb Z$ in both Types XII and XXII.

    It remains to consider Type XIII when its modification divisor is smooth. We use the notation of \cite[Lemma 5.5]{Kishimoto2005}. Let $Z=SL(2,\mathbb C)$, and let $\sigma:X\to Z$ be the affine modification with smooth divisor $D$, center $p\in D$, and exceptional divisor $E\cong\mathbb A^2$ \cite[Definition 2.2 and Lemma 5.2(2)]{Kishimoto2005}. Put $X^*=X\setminus E$ and $Z^*=Z\setminus D$. The affine modification restricts to an isomorphism $X^*\cong Z^*$. Kishimoto's calculation gives
    \[
        H_0(D,\ZZ)\cong H_1(D,\ZZ)\cong\ZZ,
        \qquad H_i(D,\ZZ)=0\quad(i>1)
    \]
    \cite[Lemma 5.5]{Kishimoto2005}. Proposition \ref{homology Gm surface} therefore gives
    \[
        M(D)\cong M(\mathbb G_m)
        \cong\ZZ\oplus\ZZ(1)[1].
    \]
    By purity \cite[Property 14.5.5]{MVW}, the two Gysin distinguished triangles form a morphism of triangles:
    \[
    \begin{tikzcd}
    M(\AA^2)(1)[1] \arrow[r] \arrow[d] 
    & M(X^*) \arrow[r] \arrow[d, "\cong"] 
    & M(X) \arrow[r] \arrow[d, "\sigma"'] 
    & M(\AA^2)(1)[2] \arrow[d]  
    \\
    M(D)(1)[1] \arrow[r] 
    & M(Z^*) \arrow[r] 
    & M(Z) \arrow[r] 
    & M(D)(1)[2]
    \end{tikzcd}
    \]
    The projection from $SL(2,\mathbb C)$ to its first column is a Zariski locally trivial $\mathbb A^1$-bundle over $\mathbb A^2\setminus\{0\}$. Therefore,
    $$M(Z)=M(SL(2,\mathbb C))\cong M(\mathbb A^2\setminus\{0\})
    \cong\mathbb Z\oplus\mathbb Z(2)[3]$$
    and, by Proposition \ref{homology Gm surface},
    $$M(D)(1)[2]\cong\mathbb Z(1)[2]\oplus\mathbb Z(2)[3].$$
    Put $A=\mathbb Z(1)[2]$ and $T=\mathbb Z(2)[3]$. With respect to these decompositions, the last morphism in the lower Gysin triangle has the form
    \[
        \delta:\mathbb Z\oplus T\longrightarrow A\oplus T,
        \qquad
        \delta=
        \begin{pmatrix}
        0&0\\
        c&n\operatorname{id}_T
        \end{pmatrix}
    \]
    for some $c\in\operatorname{Hom}(\mathbb Z,T)$ and $n\in\mathbb Z$. Indeed,
    $\operatorname{Hom}(\mathbb Z,A)=\operatorname{Hom}(T,A)=0$, while $\operatorname{End}(T)=\mathbb Z$. The Betti realization of $n\operatorname{id}_T$ is the boundary map
    $$H_3(SL(2,\mathbb C),\mathbb Z)\longrightarrow H_1(D,\mathbb Z).$$
    The upper topological Gysin sequence, together with the contractibility of $X$ and $E\cong\mathbb A^2$, gives
    $$H_1(X^*,\mathbb Z)\cong\mathbb Z,
    \qquad H_i(X^*,\mathbb Z)=0\quad(i>1).$$
    Since $X^*\cong Z^*$, the lower topological Gysin sequence shows that this boundary map is an isomorphism. Hence $n=\pm1$. The automorphism
    \[
    \mathbb Z\oplus T\longrightarrow\mathbb Z\oplus T,
    \qquad (z,t)\longmapsto(z,t-n^{-1}c(z))
    \]
    eliminates the entry $c$. After this change of decomposition, $\delta$ is the direct sum of the zero morphism $\mathbb Z\to A$ and the isomorphism $n\operatorname{id}_T:T\to T$. It follows that
    \[
        M(Z^*)\cong\mathbb Z\oplus A[-1]
        =\mathbb Z\oplus\mathbb Z(1)[1].
    \]
    Moreover, under this decomposition, the first morphism in the rotated lower Gysin triangle
    \[
        M(D)(1)[1]=A[-1]\oplus T[-1]\longrightarrow M(Z^*)
    \]
    restricts to the inclusion of the $A[-1]=\mathbb Z(1)[1]$-summand and vanishes on the $T[-1]$-summand.
    The morphism $E\to D$ factors through the center $p$. The point map $M(p)\to M(D)$ is split by the structure morphism $D\to\Spec\mathbb C$. Since
    $$\operatorname{Hom}(\mathbb Z(1)[1],\mathbb Z)=0,$$
    the above decomposition of $M(D)$ may be chosen so that $M(p)\to M(D)$ is the inclusion of its $\mathbb Z$-summand. After twisting by $(1)[1]$, the left vertical morphism in the diagram is therefore the inclusion
    \[
        M(E)(1)[1]=A[-1]\longrightarrow A[-1]\oplus T[-1]=M(D)(1)[1].
    \]
    By the commutativity of the diagram and the isomorphism $M(X^*)\cong M(Z^*)$, the first morphism in the upper Gysin triangle is the inclusion
    \[
        A[-1]\longrightarrow\mathbb Z\oplus A[-1].
    \]
    Its cone is $\mathbb Z$, and the upper Gysin triangle therefore gives $M(X)\cong\mathbb Z$. Thus $\CH^i(X)=0$ for every $i>0$, and Theorem \ref{threefold} implies that every algebraic vector bundle over $X$ is trivial.
\end{proof}
\begin{remark} \label{table1 rational}
    All the threefolds appearing in \cite[Table 1]{MS1990} and \cite[Table 1]{Kishimoto2005} are affine and rational.
    Thus also by Corollary \ref{main corollary}, every vector bundle over these threefolds is trivial.
\end{remark}
\begin{remark} \label{non tate}
    Every smooth projective rational surface is isomorphic to an iterated blow-up of either a Hirzebruch surface or the projective plane. Thus, its motive can be computed by the projective bundle formula \cite[Theorem 15.12]{MVW} and the blow-up triangle \cite[Property 14.5.4]{MVW}. Since the motives of the two base cases are Tate, the motive of a smooth projective rational surface is always Tate. However, the motive of a smooth projective rational threefold need not be Tate. For example, suppose that $S$ is an elliptic curve embedded in $\mathbb{P}^3$ and $X$ is the blow-up of $\mathbb{P}^3$ along $S$. Then $X$ is a projective rational threefold whose motive is not pure Tate.

    From the blow-up triangle,
    $$M(X) \cong M(\mathbb{P}^3) \oplus M(S)(1)[2].$$
    Since $S$ is an elliptic curve, $M(S)$ is not a pure Tate motive, and thus $M(X)$ is not pure Tate.
\end{remark}
The computations above suggest the following rational special case of Question \ref{general motive question}; compare Question \ref{question on trivial motive}.
\begin{question} \label{rational motive question}
Suppose that $X$ is a topologically contractible, smooth affine threefold over $\mathbb{C}$ such that $\bar{X}$ is rational. Is $M(X)$ trivial?
\end{question}

By the weak factorization theorem, a birational map $\mathbb P^3\dashrightarrow\bar X$ factors into a sequence of blow-ups and blow-downs along smooth centers \cite{AKMW}. In dimension $3$, the non-trivial centers may be taken to be points or smooth curves. Tracking the boundary through such a factorization may help determine the motive $M(X)$. However, a rational threefold can have a motive which is not pure Tate (Remark \ref{non tate}), and weak factorization alone does not control the motives of all boundary components.


\section{Appendix} \label{appendix}
Affine modifications provide an important source of topologically contractible varieties \cite{zaidenberglecture}. Let $\pi:\widetilde X\to X$ be the affine modification of $X$ along an irreducible hypersurface $D$ with center $C\subsetneq D$, where $C$ lies in the smooth locus of $D$. Let $\xi$ be the exceptional divisor of $\operatorname{Bl}_C(X)$ and put $\widetilde\xi=\xi\cap\widetilde X$.
Then $\pi|_{\widetilde\xi}:\widetilde\xi\to C$ is a Zariski locally trivial $\mathbb A^d$-bundle, where $d+1=\operatorname{codim}_X(C)$, and $\pi$ restricts to an isomorphism $\widetilde X\setminus\widetilde\xi\xrightarrow{\sim}X\setminus D$.
\begin{lemma} \label{connectedness of modification}
    Let $\widetilde X,X,C,D,\xi$, and $\widetilde\xi$ be as above over $\mathbb C$. If $\widetilde X$ and $D$ are $\mathbb A^1$-connected, then $X$ is $\mathbb A^1$-connected.
\end{lemma}
\begin{proof}
    Let $F/\mathbb C$ be a finitely generated field extension and let $x,y\in X(F)$. By Morel's field-valued criterion, it suffices to show that they have the same image in $\pi_0^{\mathbb A^1}(X)(F)$. This is immediate if $x,y\in D(F)$. If both lie in $(X\setminus D)(F)$, lift them through the isomorphism $\widetilde X\setminus\widetilde\xi\cong X\setminus D$ and use the $\mathbb A^1$-connectedness of $\widetilde X$.
     
     Finally, suppose that $x\in(X\setminus D)(F)$ and $y\in D(F)$. Lift $x$ to $\widetilde x\in\widetilde X(F)$. Choose $\widetilde x_0\in\widetilde\xi(\mathbb C)$ and set $x_0=\pi(\widetilde x_0)\in C(\mathbb C)$. The $\mathbb A^1$-connectedness of $\widetilde X$ identifies $x$ with $x_0$ in $\pi_0^{\mathbb A^1}(X)(F)$, while that of $D$ identifies $x_0$ with $y$. Hence $x$ and $y$ have the same image, as required.
\end{proof}
\begin{remark}
    The converse is false. For example, the tom Dieck--Petrie surfaces are affine modifications of the affine plane, but they are not $\mathbb{A}^1$-connected.
\end{remark}
\begin{example}
    \begin{enumerate}
        \item Let $X$ be a tom Dieck--Petrie surface; for example,
        \[
        X=\{z^2x^3+3zx^2+3x-zy^2-2y=1\}\subset\mathbb A^3_{\mathbb C}.
        \]
        Put $Y=X\times_{\mathbb C}\mathbb A^1_{\mathbb C}$. The surface $X$ is topologically contractible and has logarithmic Kodaira dimension $1$ \cite{Asok_motive}, whereas the cylinder $Y$ has logarithmic Kodaira dimension $-\infty$. The surface $X$ contains a unique affine line $\gamma=X\cap\{z=0\}$. Let $\widetilde Y$ be the affine modification associated with $(Y,p=(1,1,0,0),\Gamma=\gamma\times\mathbb A^1)$:
        $$\Tilde{Y} = \Sigma_{p, \Gamma} (Y).$$ 
        The variety $\Tilde{Y}$ is a smooth affine threefold, which is birational to $X \times_{\mathbb{C}} \mathbb{A}^1_{\mathbb{C}}$.
         \begin{remark}
           \begin{enumerate}
               \item Since $X$ and $\gamma$ are topologically contractible, $\Tilde{Y}$ is also topologically contractible.
               \item Since $X$ is rational, $\Tilde{Y}$ is also rational.
               \item Since $\Gamma$ is smooth, the Voevodsky motive $\mathbf{M}(\Tilde{Y}) \cong \mathbb{Z}$ \cite[Proposition 4]{Asok_motive}. Thus, in particular, all algebraic vector bundles over $\Tilde{Y}$ are trivial.
               \item Since $\Gamma$ is smooth and $X$ is not $\mathbb{A}^1$-contractible, \cite[Theorem 2.14]{DPO} shows that $\Tilde{Y}$ is not $\mathbb{A}^1$-contractible. However, by \cite[Theorem 3.2]{DPO}, $\Tilde{Y}$ is stably $\mathbb{A}^1$-contractible. The triviality of algebraic vector bundles over $\Tilde{Y}$ follows also from the motivic computation in the preceding item.
               \item The threefold $\Tilde{Y}$ has negative logarithmic Kodaira dimension. We can prove it this way. The ideal $I$ in $\mathcal{O}(Y)$ corresponding to the point $p \in Y$ is generated by the regular sequence $(\bar{z}, \overline{x-1}, \bar{w})$. Thus $\Tilde{Y}$ is given by the following system of equations in $Y \times_{\mathbb{C}} \mathbb{A}^2_{\mathbb{C}}$ (here, $Y \subset \Spec \ \mathbb{C}[x, y, z, w]$) as
               $$zy_1 = x - 1 \text{ and } zy_2 = w.$$
               The open subset $\Tilde{Y} \cap \{z \neq 0\}$ is isomorphic to $(X \cap \{z \neq 0\}) \times_{\mathbb{C}} \mathbb{A}^1_{\mathbb{C}}$. Thus, $\bar{\kappa}(\Tilde{Y}) = - \infty$.
               \item Here the modification divisor $\Gamma\cong\mathbb A^2$ is $\mathbb A^1$-connected, whereas the base $Y=X\times\mathbb A^1$ is not. The contrapositive of Lemma \ref{connectedness of modification} therefore shows that $\widetilde Y$ is not $\mathbb A^1$-connected.
               \item Suppose that we take $X$ to be the Ramanujam surface or a tom Dieck--Petrie surface, let $p = (q, 0) \in Y$ (where $Y = X \times_{\mathbb{C}} \mathbb{A}^1_{\mathbb{C}}$ and $q \in X(\mathbb{C})$), and set $\Tilde{Y} = \Sigma_{p, X} (Y)$. Then all the remarks above remain true except possibly (vi): we do not know whether $\Tilde{Y}$ is $\mathbb{A}^1$-connected in these cases.
               \end{enumerate}
    \end{remark}
\item Suppose that $X \subset \Spec \ \mathbb{C}[x, y, z]$ is given by the Brieskorn--Pham polynomial
    $$x^a + y^b+ z^c = 0, \qquad \text{where }a, b, c \geq 2 \text{ are pairwise coprime integers.}$$
    We may assume that $b$ is odd. The variety $X$ is normal and has a unique singularity at the origin. Consider the affine modification $\Tilde{X}$ of the triple $(\mathbb{A}^3_{\mathbb{C}}, p = (1, -1,0), X)$, so
    $$\Tilde{X} = \Sigma_{p, X}(\mathbb{A}^3_{\mathbb{C}}).$$
 Then $\Tilde{X}$ is a smooth rational affine threefold, and by \cite[Example 2.2]{Zaidenbergmodification}, it can be realized as the hypersurface in $\Spec \ \mathbb{C}[x,y,z,w]$ given by
 $$\frac{(xw+1)^a + (yw-1)^b + (zw)^c - w}{w} = 0,$$
 for simplicity, if we take $(a, b, c) = (2, 3, 5)$, then $\Tilde{X}$ is given by
 $$\{z^5w^4+ y^3w^2+ x^2w- 3y^2w+2x+3y = 1\} \subset \Spec \ \mathbb{C}[x, y, z, w].$$
 \begin{remark}\label{modificationremark}
 \begin{enumerate}
     \item The threefold $\Tilde{X}$ is topologically contractible. Indeed, the hypersurface in $\mathbb{A}^3_{\mathbb{C}}$ given by the Brieskorn polynomial $x^a+y^b+z^c$ admits a $\mathbb{G}_m$-action with unique fixed point, given by
     $$\lambda\cdot(x, y, z) = (\lambda^{bc}x, \lambda^{ac} y, \lambda^{ab}z).$$
     The $\mathbb{G}_m$-action map extends to the explicit homotopy $\mathbb{A}^1_{\mathbb{C}} \times_{\mathbb{C}} X \to X$, which is the null homotopy. So $\Tilde{X}$ is topologically contractible.
     \item Every algebraic vector bundle on $\Tilde X$ is trivial by Corollary \ref{contractible rational vector bundles}. It is not known in general whether $\Tilde X$ is stably $\mathbb A^1$-contractible or whether $M(\Tilde X)\cong\mathbb Z$. If $(a,b,c)=(2,3,5)$ or $c=mab+1$ for some $m>0$, however, $\Tilde X$ is stably $\mathbb A^1$-contractible \cite[Corollaries 3.7 and 3.8]{DPO}.
      \end{enumerate}
 \end{remark}
\begin{question}
    \begin{enumerate}
        \item Is $\Tilde{X}$ $\mathbb{A}^1$-connected? Does $\Tilde{X}$ have negative logarithmic Kodaira dimension? The variety $\Tilde{X}$ admits a morphism
        $$\phi: \Tilde{X} \to \mathbb{P}^2_{\mathbb{C}} \text{ given by } (x,y,z,w) \mapsto [(xw+1)^a:(yw-1)^b:(zw)^c].$$
        \item Can $\Tilde{X}$ be $\mathbb{A}^1$-contractible? We cannot use \cite[Theorem 2.14]{DPO} here, since the divisor is not smooth.
    \end{enumerate}
\end{question}
    \end{enumerate}
   
\end{example}


    \bibliographystyle{alpha}
	\bibliography{references}

\end{document}